\documentclass[a4paper]{amsart}
\usepackage{amssymb}
\usepackage{mathtools}
\usepackage[hidelinks]{hyperref}
\usepackage[numbered]{bookmark}     
\usepackage[maxbibnames=10, 
  url=false, 
  doi=false, 
  date=year,
  isbn=false]{biblatex}
\DeclareFieldFormat{postnote}{#1}
\DeclareFieldFormat{multipostnote}{#1}
\usepackage{enumitem}               
\usepackage{tikz-cd}

\newtheorem{thm}{Theorem}[section]
\newtheorem{definition}[thm]{Definition}
\newtheorem{lemma}[thm]{Lemma}
\newtheorem{corollary}[thm]{Corollary}
\newtheorem{proposition}[thm]{Proposition}
\newtheorem{example}[thm]{Example}

\newtheorem{introthm}{Theorem}

\newtheorem{introprop}[introthm]{Proposition}

\newcommand{\Ha}{\mathrm{H}}

\newcommand{\C}{\mathbb{C}}
\newcommand{\N}{\mathbb{N}}
\newcommand{\Z}{\mathbb{Z}}

\title[C*-Algebras Generated by Two Partitions of Unity]{Classification of Certain C*-Algebras Generated by Two Partitions of Unity}
\author{Bj\"orn Sch\"afer}
\email{bschaefer@math.uni-sb.de}
\address{Department of Mathematics, Saarland University, D-66123 Saarbr\"ucken, Germany}
\date{\today}

\AtEveryBibitem{\clearlist{language}} 

\begin{document}

\begin{abstract}
    We study C*-algebras generated by two partitions of unity subject to orthogonality relations governed by a bipartite graph which we also call ``bipartite graph C*-algebras''.
    These algebras generalize at the same time the C*-algebra $C^\ast(p,q)$ generated by two projections and the hypergraph C*-algebras of Trieb, Weber and Zenner.
    We describe alternative universal generators of bipartite graph C*-algebras and study partitions of unity in ``generic position'' associated to a bipartite graph.
    As a main result, we prove that bipartite graph C*-algebras are completely classified by their one- and two-dimensional irreducible representations which provides a first step towards a classification of the more general hypergraph C*-algebras.
\end{abstract}

\maketitle

\setcounter{tocdepth}{1}    
\bookmarksetup{depth=2}     
\tableofcontents

\section{Introduction}

Let $P$ and $Q$ be projections on some Hilbert space $\mathcal{H}$. Famously, Halmos described the relation between $P$ and $Q$ in his Two Projection Theorem \cite{halmos_two_1969}:
If the two projections are in a particular ``generic position'',
then $P$ and $Q$ are unitarily equivalent to
\begin{align*}
    \begin{pmatrix}
        I & 0 \\
        0 & 0
    \end{pmatrix}
    \quad \text{ and } \quad
    \begin{pmatrix}
        C^2 & CS \\
        CS & S^2
    \end{pmatrix},
\end{align*}
where $C$ and $S$ are positive contractions on some Hilbert space $\mathcal{K}$ with $S^2 + C^2 = I$, $I \in B(\mathcal{H})$ being the identity operator. This theorem has various precursors of which we only mention \cite{krein_defect_1948, dixmier_position_1948, davis_separation_1958}. A more thorough account of the history is provided by B\"ottcher and Spitkovsky in \cite{bottcher_gentle_2010}.

From a more abstract point of view, Pedersen considered the universal $C^\ast$-algebra $C^\ast(p,q)$ that is generated by two arbitrary projections $p$ and $q$. This algebra can be described explicitly as
\begin{align*}
    C^*(p,q) \cong \{ f \in C([0,1], M_2) \mid f(0) \text{ and } f(1) \text{ are diagonal matrices} \},
\end{align*}
where $M_2$ is the algebra of complex $2 \times 2$ matrices \cite{pedersen_measure_1968}.
Different proofs of this fact can be found e.g. in \cite{power_hankel_1982} (based on Halmos' theorem), \cite{roch_algebras_1988} (as a special case of the more general situation of a Banach algebra generated by idempotents), or \cite{raeburn_c-algebra_1989} (using the Mackey machine). Another proof can be found in \cite{vasilevski_algebra_1981}, see also \cite{bottcher_gentle_2010}.

In this paper, we investigate a class of $C^\ast$-algebras associated to bipartite graphs which generalizes $C^\ast(p,q)$. More precisely, given a bipartite graph $G = (U, V, E)$ with disjoint vertex sets $U$ and $V$ and edge set $E \subset \{\{u,v\} \mid u \in U, v \in V\}$, we consider the universal C*-algebra $C^\ast(G)$ that is generated by projections $p_x$ with $x \in U \cup V$ satisfying the two relations
\begin{align}
    \sum_{u \in U} p_u &= 1 = \sum_{v \in V} p_v,        \tag{\ref{bga::eq:partition_relation}} \\
    p_u p_v &= 0 \; \text{ if } \{u, v\} \not \in E.  \tag{\ref{bga::eq:orthogonality_relation}}
\end{align}
Thus, the families $\{p_u\}_{u \in U}$ and $\{p_v\}_{v \in V}$ consist of pairwise orthogonal projections which add up to the unit of $C^\ast(G)$. Such a family is called a partition of unity and therefore $C^\ast(G)$ is generated by two partitions of unity. We call $C^\ast(G)$ the \emph{bipartite graph C*-algebra} associated to $G$. As a particular case, the universal C*-algebra $C^\ast(p,q)$ of two projections is retained as $C^\ast(K_{2,2})$, where $K_{2,2}$ is the complete bipartite graph with two vertices on each side.

Our interest in these algebras stems from their connection to the recently introduced hypergraph $C^\ast$-algebras which have first been investigated by Trieb, Weber and Zenner in \cite{trieb_hypergraph_2024}. Hypergraph C*-algebras are a novel generalization of graph C*-algebras and not much is known about them. Open questions include their classification or which hypergraph C*-algebras are nuclear. The latter question was studied by Moritz Weber and the author in \cite{schafer_nuclearity_2024}. There, a partial result was obtained: In order to tell which hypergraph C*-algebras are nuclear it suffices to know which \emph{undirected} hypergraph $C^\ast$-algebras are nuclear. In fact, the latter are exactly the algebras we study in the present paper (see Proposition \ref{bga::prop:connection_to_hypergraph_algebras}), though we prefer the language of bipartite graphs for the sake of clarity.

As a first result of this work, we obtain alternative generators of the bipartite graph C*-algebra $C^\ast(G)$ where the universal generators are not projections associated to the vertices but contractions associated to the edges of $G$. This offers a different perspective on these C*-algebras and might be useful for future investigations.

\begin{introprop}[Proposition \ref{bga::proposition:alternative_definition}]
  Let $G = (U, V, E)$ be a bipartite graph. Then $C^\ast(G)$ is the universal $C^\ast$-algebra generated by a family of elements $(x_e)_{e \in E}$ satisfying
  \begin{align}
      x_e^\ast x_f &= 0 & & \text{ if } e \cap f \cap U = \emptyset,  \tag{\ref{bga::eq:gc1}} \\
      x_e x_f^\ast &= 0 & & \text{ if } e \cap f \cap V = \emptyset,  \tag{\ref{bga::eq:gc2}} \\
      \left( \sum_{e \in E} x_e^\ast \right) x_f &= x_f,              \tag{\ref{bga::eq:gc3}} \\
      x_e \left( \sum_{f \in E} x_f^\ast \right) &= x_e,              \tag{\ref{bga::eq:gc4}}
  \end{align}
  for all edges $e, f \in E$, respectively. In particular, the $x_e$ are contractions which satisfy $x_e x_e^\ast x_e = x_e^2$.
\end{introprop}

Next, we investigate how the classical notion of projections in ``generic position'' can be adapted to our situation. In \cite{vasilevski_algebra_1981} Vasilevski generalized Halmos' notion of two projections in generic position to the case of two partitions of unity of the form $\{P, 1-P\}$ and $\{Q_1, \dots, Q_n\}$ where $P$ and $Q_i$ are projections on some Hilbert space $\mathcal{H}$. His results carry over to a $G$-projection family for a bipartite graph $G$ as follows.

\begin{introthm}[Theorem \ref{genpos::thm:projections_in_generic_position}]
  Every $G$-projection family $(P_x)_{x \in U \cup V}$ in generic position on some Hilbert space $\mathcal{H}$ associated to a connected bipartite graph $G$ is (up to unitary equivalence) of the form
  \begin{align*}
    P_u &= (C_{u v_1}^\ast C_{u v_2})_{v_1, v_2 \in V} \in M_{V}(\mathcal{B}(\mathcal{K})), \\
    P_v &= (\delta_{v_1 v} \delta_{v_2 v})_{v_1, v_2 \in V} \in M_{V}(\mathcal{B}(\mathcal{K})),
  \end{align*}
  for operators $(C_{uv})_{u \in U, v \in V}$ on a Hilbert space $\mathcal{K}$ that satisfy particular relations including $C_{uv} = 0$ if $\{u,v \} \not \in E$.
\end{introthm}

The notion of a $G$-projection family in generic position is made precise in Definition \ref{genpos::def::generic_position} and the conditions imposed on the operators $C_{uv}$ can be found in the statement of Theorem \ref{genpos::thm:projections_in_generic_position}. They generalize the condition $C^2 + S^2 = I$ from Halmos' Theorem and incorporate the graph structure by asking $C_{uv} = 0$ whenever $\{u,v\} \not \in E$.

Finally, the main result of this work is a classification of bipartite graph C*-algebras which answers the question: When is $C^\ast(G) \cong C^\ast(G^\prime)$ for two bipartite graphs $G$ and $G^\prime$. For that we consider the subspace $\mathrm{Spec}_{\leq 2}(C^\ast(G))$ of the spectrum of $C^\ast(G)$ which contains only the (equivalence classes of) one- and two-dimensional irreducible representations.  The structure of this space can be easily read off from the bipartite graph $G$. In fact, the one-dimensional irreducible representations correspond to edges of the bipartite graph, while the two-dimensional irreducible representations correspond to subgraphs that are isomorphic to $K_{2,2}$, see Lemma \ref{reps::lemma:representations_vs_edges_and_subgraphs} and Lemma \ref{reps::lemma:Spec<=2(C*G)}. Then we obtain the following theorem.

\begin{introthm}[Theorem \ref{class::thm:classification_of_bipartite_graph_algebras}]
    \label{intro::thm:classification_of_bipartite_graph_algebras}
    We have 
    \begin{align*}
        C^\ast(G) \cong C^\ast(G^\prime) 
            \quad \Leftrightarrow \quad 
        \mathrm{Spec}_{\leq 2}(C^\ast(G)) \cong \mathrm{Spec}_{\leq 2}(C^\ast(G^\prime)).
    \end{align*}
\end{introthm}

Recall that bipartite graph C*-algebras are special cases of hypergraph C*-algebras. Hence, Theorem \ref{intro::thm:classification_of_bipartite_graph_algebras} provides a first step towards a classification of hypergraph C*-algebras. As mentioned above, another open question for hypergraph C*-algebras is which of them are nuclear. It is a simple observation that a bipartite graph C*-algebra $C^\ast(G)$ is not nuclear whenever $K_{2,3} \subset G$, see Corollary \ref{bga::corollary:K23_implies_not_nuclear}. However, it is not known if the converse holds. In upcoming work we investigate a special class of bipartite graphs $G$ with $K_{2,3} \not \subset G$, namely the hypercubes $Q_n$. For these graphs we obtain an explicit description of $C^\ast(G)$ as algebra of continuous functions and, thus, we can show that they are nuclear.

\subsection{Outline}

In \textbf{Section \ref{sec::bipartite_graph_algebras}} we introduce bipartite graph C*-algebras and discuss some of their properties including a set of alternative generators and the connection to hypergraph C*-algebras. The next \textbf{Section \ref{sec::generic_position}} extends results of Vasilevski \cite{vasilevski_c-algebras_1998} on projections in generic position to $G$-projection families associated to a bipartite graph $G$. In \textbf{Section \ref{sec::one_and_two_dimensional_representations}} we describe the space of one- and two-dimensional irreducible representations of a bipartite graph C*-algebra $C^\ast(G)$ in terms of the combinatorial structure of $G$. Finally, in \textbf{Section \ref{sec::classification}} we prove the classification result from Theorem \ref{intro::thm:classification_of_bipartite_graph_algebras}.

\subsection{Acknowledgements}

The author would like to thank his supervisor Moritz Weber for many helpful discussions and comments on earlier versions of this article. This work is part of the author's PhD thesis and a contribution to the SFB-TRR 195.

\section{Bipartite graph C*-algebras}
\label{sec::bipartite_graph_algebras}

In this section, we introduce the main object of our investigations: a class of C*-algebras that are generated by two universal finite partitions of unity which satisfy certain orthogonality relations. The precise relations are given by a bipartite graph, and this is why we call the obtained C*-algebras ``bipartite graph C*-algebras''. 

\subsection{Bipartite graphs}

First, let us recall the notion of a bipartite graph and settle some notation. Throughout this paper all graphs are finite.

\begin{definition}
    A bipartite graph $G$ consists of two finite vertex sets $U$ and $V$ together with an edge set $E \subset \{\{u,v\}\mid u \in U, v \in V\}$. We write $G = (U, V, E)$.
\end{definition}
 
For two vertices in a bipartite graph $G = (U, V, E)$ let us write $u \sim v$ if $\{u,v\} \in E$ and let $\mathcal{N}(u) := \{v \in V\mid v \sim u\}$ ($\mathcal{N}(v) := \{u \in U\mid v \sim u\}$) be the neighbors of $u$ ($v$). A path $\mu$ in a bipartite graph $G$ is a finite sequence of vertices $v_1 \dots v_n$ such that $v_i \sim v_{i+1}$ holds for all $i < n$. For every $i < n$, we say that $v_i$ is contained in $\mu$. We set $s(\mu) = v_1$ and $r(\mu) = v_n$. For two paths $\mu=v_1 \dots v_m$ and $\nu = u_1 \dots u_n$ with $r(\mu) = s(\nu)$ we write $\mu \nu$ for the combined path $v_1 \dots v_m u_1 \dots u_n$.

\begin{example}
    For $m, n \in \N$, the complete bipartite graph $K_{m,n} = (U, V, E)$ of order $(m,n)$ is given by 
    \begin{align*}
        U &= \{u_1, \dots, u_m\}, \\
        V &= \{v_1, \dots, v_n\}, \\
        E &= \{\{u_i, v_j\}\mid i = 1, \dots, m, j = 1, \dots, n\}.
    \end{align*}
    In particular, the graph $K_{2,2}$ is described by the sketch below.
    
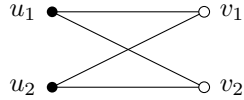
\begin{center}
        \begin{tikzpicture}
            \draw (-1,2) -- (1, 2);
            \draw (-1,1) -- (1, 1);
            \draw (-1,2) -- (1, 1);
            \draw (-1,1) -- (1, 2);
            \node[label=180:$u_1$, fill=black, circle, inner sep=.05cm] at (-1, 2) {};
            \node[label=180:$u_2$, fill=black, circle, inner sep=.05cm] at (-1, 1) {};
            \node[label=0:$v_1$, fill=white, draw=black, circle, inner sep=.05cm] at (1, 2) {};
            \node[label=0:$v_2$, fill=white, draw=black, circle, inner sep=.05cm] at (1, 1) {};
        \end{tikzpicture}
    \end{center}
\end{example}

\begin{definition}
    \label{bga::def:bipartite_graph_operations}
    Let $G = (U, V, E)$ and $G^\prime = (U^\prime, V^\prime, E^\prime)$ be bipartite graphs.
    \begin{enumerate}
        \item $G^\prime$ is a subgraph of $G$, written $G^\prime \subset G$, if 
        \begin{itemize}
            \item $U^\prime \subset U$,
            \item $V^\prime \subset V$, and 
            \item $E^\prime \subset \{\{u, v\} \in E\mid u \in U^\prime, v \in V^\prime\}$,
        \end{itemize}
        or if the same holds after swapping $U^\prime$ and $V^\prime$.
        \item $G^\prime$ is the subgraph of $G$ induced by the set $\{x_1, \dots, x_n\} \subset U \cup V$, written $G^\prime = G(x_1, \dots, x_n)$ if
        \begin{itemize}
            \item $U^\prime = U \cap \{x_1, \dots, x_n\}$,
            \item $V^\prime = V \cap \{x_1, \dots, x_n\}$, and
            \item $E^\prime = \{e \in E\mid e \subset \{x_1, \dots, x_n\}\}$.
        \end{itemize}
        Similarly, $G^\prime$ is the subgraph of $G$ induced by the set $\{e_1, \dots, e_n\} \subset E$, written $G^\prime = G(e_1, \dots, e_n)$ if
        \begin{itemize}
            \item $U^\prime = U \cap \left( \bigcup \{e_1, \dots, e_n\} \right)$,
            \item $V^\prime = V \cap \left( \bigcup \{e_1, \dots, e_n\} \right)$,
            \item $E^\prime = \{e_1, \dots, e_n\}$.
        \end{itemize}
        \item $G^\prime$ is isomorphic to $G$, written $G^\prime \cong G$, if there are two bijective maps $\varphi: U \to U^\prime$ and $\psi: V \to V^\prime$ such that 
        $$ E^\prime = \{\{\varphi(u), \psi(v)\}\mid u \in U, v \in V, \{u, v\} \in E\}, $$
        or if the same holds after swapping $U^\prime$ and $V^\prime$.
    \end{enumerate}
\end{definition}

\subsection{Definition of bipartite graph C*-algebras}

Let us now introduce bipartite graph C*-algebras.

\begin{definition} 
    \label{bipartite_graph_algebras:definition}
    Given a bipartite graph $G = (U, V, E)$ let $C^\ast(G)$ be the universal $C^\ast$-algebra generated by a family of projections $(p_x)_{x \in U \cup V}$ subject to the following relations:
    \begin{align}
        \sum_{u \in U} p_u &= 1 = \sum_{v \in V} p_v,        \tag{GP1} \label{bga::eq:partition_relation} \\
        p_u p_v &= 0 \; \text{ if } \{u, v\} \not \in E.     \tag{GP2} \label{bga::eq:orthogonality_relation}
    \end{align}
    
    We call a family of projections $(P_x)_{x \in U \cup V} \subset B(\mathcal{H})$ on a Hilbert space $\mathcal{H}$ a $G$-projection family if they satisfy the relations \eqref{bga::eq:partition_relation} and \eqref{bga::eq:orthogonality_relation}.
\end{definition}

Recall that the universal $C^\ast$-algebra generated by elements $(x_i)_i$ subject to relations $(R_j)_j$ is the unique $C^\ast$-algebra $A$ generated by elements $x_i$ with the following universal property: Whenever another $C^\ast$-algebra $B$ is generated by elements $(y_i)_i$ subject to the same relations $(R_j)_j$, then there is a unique $\ast$-homomorphism $\varphi: A \to B$ such that $\varphi(x_i) = y_i$ for all $i$. Depending on the generators and relations this $C^\ast$-algebra need not exist. In the special case above, the relations include that the $p_x$ are projections, and in this case the existence of the universal $C^\ast$-algebra is guaranteed. For a reference see e.g. \cite[Section II.8.3]{blackadar_operator_2006}.

\begin{example}
    \label{bga::example:C_nm_as_bipartite_graph_algebra}
    Consider the complete bipartite graph $K_{2,2}$ of order $(2,2)$. Then $C^\ast(K_{2,2}) = C^\ast(\Z_2 \ast \Z_2)$ is the universal unital $C^\ast$-algebra $C^\ast(p,q)$ generated by two projections. Indeed, the two left vertices of $K_{2,2}$ are associated to the projections $p$ and $1-p$, while the two right vertices are associated to the projections $q$ and $1-q$. As $K_{2,2}$ is a complete graph there are no orthogonality requirements on these two partitions of unity. More generally, we have $C^\ast(K_{n,m}) \cong \C^n *_\C \C^m$ for all $n,m \in \N \setminus \{0\}$.
\end{example}

Using the relations \eqref{bga::eq:partition_relation} and \eqref{bga::eq:orthogonality_relation} one easily sees that a dense subset of $C^\ast(G)$ is spanned by elements associated to paths in $G$.

\begin{proposition} 
    \label{prop::bipartite_graph_algebras:definition:dense_subset}
    Let $G$ be a bipartite graph as in Definition \ref{bipartite_graph_algebras:definition}, and for every path $\mu = x_1 \dots x_n$ in $G$, write $p_\mu := p_{x_1} \cdots p_{x_n}$ for the associated element in $C^\ast(G)$. Then the elements $p_\mu$ span a dense subset of $C^\ast(G)$. 
\end{proposition}

\begin{proof}
    Evidently, a dense subset of $C^\ast(G)$ is spanned by arbitrary products of the form $a = p_{x_1} p_{x_2} \cdots p_{x_n}$ with $x_1, \dots, x_n \in U \cup V$. Since the $p_{x_i}$ are projections we may assume without loss of generality that $x_i \neq x_{i+1}$ holds for all $i < n$. 
    
    To prove the statement, it suffices to observe that $a$ vanishes whenever $x_1 \dots x_n$ is not a path in $G$. Indeed, whenever $x_i$ and $x_{i+1}$ for $i < n$ are both in $U$ or both in $V$, then $p_{x_i} p_{x_{i+1}} = 0$, for $x_i \neq x_{i+1}$ (by assumption) and $(p_u)_{u \in U}$ (resp. $(p_v)_{v \in V}$) is a family of pairwise orthogonal projections by \eqref{bga::eq:partition_relation}. Thus, $a$ vanishes if the $x_i$ are not alternatingly from $U$ and $V$. Finally, assume without loss of generality that $x_i \in U$ and $x_{i+1} \in V$ for $i < n$. Then $p_{x_i} p_{x_{i+1}} = 0$ unless $\{x_i, x_{i+1}\} \in E$ by \eqref{bga::eq:orthogonality_relation}. This proves the claim.
\end{proof}

In a similar way as the last proposition, one obtains the next lemma which will be useful later.

\begin{lemma}
    \label{bga::lemma:nontrivial_p_x2 p_y p_x2_and_K22_subgraph}
    Let $G = (U, V, E)$ be a bipartite graph, and let $x_1, x_2, y \in U \cup V$ be three distinct vertices. Then 
    \begin{align*}
        p_{x_1} p_y p_{x_2} \neq 0
        \quad \implies \quad 
        \{y\} \subsetneq \mathcal{N}(x_1) \cap \mathcal{N}(x_2).
    \end{align*}
    In particular, in this case there must be a fourth vertex $y^\prime$ such that 
    $$
    G(x_1, x_2, y, y^\prime) \cong K_{2,2}.
    $$
\end{lemma}

\begin{proof}
    Assume $p_{x_1} p_y p_{x_2} \neq 0$ and assume that the statement on the right-hand side is not true, i.e. $\mathcal{N}(x_1) \cap \mathcal{N}(x_2)$ contains at most the vertex $y$. If $y$ is not contained, then one has directly $p_{x_1} p_y = 0$ or $p_y p_{x_2} = 0$ thanks to \eqref{bga::eq:orthogonality_relation} and the fact that $x_1, y, x_2$ are distinct. Otherwise, assume without loss of generality $x_1, x_2 \in V$ and $y \in U$. Then one observes 
    \begin{align*}
        p_{x_1} p_y p_{x_2} = p_{x_1} \left( \sum_{u \in U} p_u \right) p_{x_2} = p_{x_1} p_{x_2}= 0.
    \end{align*}
    Indeed, for all $u \in U \setminus \{y\}$ it is $y \not \in \mathcal{N}(x_1)$ or $y \not \in \mathcal{N}(x_2)$. Hence, for these $u$ \eqref{bga::eq:orthogonality_relation} yields $p_{x_1} p_u p_{x_2} = 0$. The latter equality follows directly from \eqref{bga::eq:partition_relation} since $\sum_{u \in U} p_u = 1$. Finally, we have $p_{x_1} p_{x_2} = 0$, for $x_1 \neq x_2$ and $(p_v)_{v \in V}$ is a family of pairwise orthogonal projections by \eqref{bga::eq:partition_relation}.
\end{proof}

Let us discuss how bipartite graph C*-algebras behave with respect to subgraphs.

\begin{proposition}
    \label{bga::proposition:algebra_of_subgraphs}
    Let $G = (U, V, E)$ be a bipartite graph and let $H = (U^\prime, V^\prime, E^\prime)$ be a subgraph. Then $C^\ast(H) \cong C^\ast(G) / (S)$, where $(S) \subset C^\ast(G)$ is the ideal generated by the elements $p_x$ with $x \not \in U^\prime \cup V^\prime$.
\end{proposition}

\begin{proof}
    Let $(\hat{p}_x)_{x \in U^\prime \cup V^\prime}$ be the generators of $C^\ast(H)$ and set
    \begin{align*}
        P_x := \begin{cases}
            \hat{p}_x, &\text{if } x \in U^\prime \cup V^\prime, \\
            0,   &\text{otherwise},
            \end{cases}
    \end{align*}
    for all $x \in U \cup V$. Then $(P_x)_{x \in U \cup V}$ is a $G$-projection family, and the universal property of $C^\ast(G)$ yields a $\ast$-homomorphism $\pi: C^\ast(G) \to C^\ast(H)$ such that $\pi(p_x) = P_x$ for all $x \in U \cup V$. Evidently, $\pi$ is surjective and its kernel is $(S)$.
\end{proof}

\begin{corollary}
    \label{bga::corollary:K23_implies_not_nuclear}
    Let $G$ be a bipartite graph with $K_{2,3} \subset G$. Then $C^\ast(G)$ is neither nuclear nor exact.
\end{corollary}

\begin{proof}
    Combining the previous Proposition \ref{bga::proposition:algebra_of_subgraphs} and Example \ref{bga::example:C_nm_as_bipartite_graph_algebra}, we see that the bipartite graph C*-algebra $\C^2 *_\C \C^3$ is a quotient of $C^\ast(G)$ as soon as $K_{2,3} \subset G$ holds. It is well-known that $\C^2 *_\C \C^3$ is neither nuclear nor exact and both properties are preserved under taking quotients. Therefore, $C^\ast(G)$ is neither nuclear nor exact as well.
\end{proof}

We end the section with a technical looking lemma that will be useful later in the classification of bipartite graph C*-algebras.

\begin{lemma}
    \label{bga::lemma:loose_edge_gives_direct_summand_C}
    Let $G = (U, V, E)$ be a bipartite graph and assume that $e = \{u, v\} \in E$ is not contained in a subgraph of $G$ that is isomorphic to $K_{2,2}$. Then, we have
    \begin{align*}
        C^\ast(G) \cong C^\ast(G^\prime) \oplus \C,
    \end{align*}
    where $G^\prime$ is obtained by deleting the edge $e$ from $G$. The isomorphism sends $p_x \in C^\ast(G)$ to the corresponding element $(p_x, 0) \in C^\ast(G^\prime) \oplus \C$ if $x \not \in e$, and otherwise to $(p_x,1)$.
\end{lemma}

\begin{proof}
    In view of Lemma \ref{bga::lemma:nontrivial_p_x2 p_y p_x2_and_K22_subgraph} one has for all $v^\prime \in V \setminus \{v\}$ and $u^\prime \in U \setminus \{u\}$
    \begin{align*}
        p_v p_u p_{v^\prime} = 0,
        \quad 
        p_u p_v p_{u^\prime} = 0.
    \end{align*}
    Using this it is not hard to show that $p_u p_v = p_v p_u$ is a projection, and it is orthogonal to all $p_x$ with $x \not \in \{u,v\}$.
    Further, it is easily checked that the family $(p_x^\prime)_{x \in U \cup V}$ defined by
    \begin{align*}
        p^\prime_x := \begin{cases}
            p_x, &\text{if } x \not \in e, \\
            p_x - p_u p_v, &\text{if } x \in e,
        \end{cases}
    \end{align*}
    is a universal $G^\prime$-projection family in $C^\ast(G)$. Thus, it follows
    \begin{align*}
        C^\ast(G) \cong C^\ast(p^\prime_x\mid x \in U \cup V) \oplus \C p_{u}p_v \cong C^\ast(G^\prime) \oplus \C,
    \end{align*}
    where the isomorphism maps $p_u p_v \mapsto (0, 1)$.
\end{proof}

\subsection{Connection to hypergraph C*-algebras}

In this section, we discuss the connection between bipartite graph C*-algebras and hypergraph C*-algebras. Hypergraph C*-algebras were introduced by Trieb, Weber and Zenner in \cite{trieb_hypergraph_2024}. With the aim of studying nuclearity of hypergraphs, in \cite{schafer_nuclearity_2024} their definition was slightly extended to so-called undirected hypergraphs. Using that definition the following result was obtained: For every hypergraph $\Ha\Gamma$ one can construct and undirected hypergraph $\Ha\Delta$ such that $C^\ast(\Ha\Gamma)$ is nuclear if and only if $C^\ast(\Ha\Delta)$ is nuclear, see \cite[Theorem 2.5]{schafer_nuclearity_2024}. We will show below that undirected hypergraph C*-algebras are nothing else than bipartite graph C*-algebras, and this explains our interest in the latter. More generally, we believe that a good understanding of bipartite graph C*-algebras is a crucial prerequisite for studying hypergraph C*-algebras.

Let us discuss the precise definition of undirected hypergraph C*-algebras. An undirected hypergraph $\Ha\Gamma$ consists of a vertex set $E^0$, an edge set $E^1$ 
and a ``source map'' $s: E^1 \to \mathcal{P}(E^0) \setminus \{\emptyset\}$. 
The associated hypergraph C*-algebra $C^\ast(\Ha\Gamma)$ is then the universal C*-algebra generated by pairwise orthogonal projections $(p_v)_{v \in E^0}$ and partial isometries $(s_e)_{e \in E^1}$ satisfying for all $e, f \in E^1$ and $v \in E^0$, resp.,
\begin{align*}
    s_e^\ast s_f &= \delta_{ef} s_e, \tag{HR1} \\
    s_e s_e^\ast &\leq \sum_{v \in s(e)} p_v, \tag{HR2} \\
    p_v &\leq \sum_{e \in E^1: v \in s(e)} s_e s_e^\ast, \tag{HR3}
\end{align*}
see \cite[Definition 2.2]{schafer_nuclearity_2024}.

\begin{proposition}
    \label{bga::prop:connection_to_hypergraph_algebras}
    Let $\Ha\Gamma = (E^0, E^1, s)$ be an undirected hypergraph and let the bipartite graph $G = (U, V, E)$ be given by
    \begin{align*}
        U &= E^0, & V &= E^1, & E &= \{\{v, e\} \in E^0 \times E^1\mid v \in s(e)\}.
    \end{align*}
    Then $C^\ast(\Ha\Gamma) \cong C^\ast(G)$ as $C^\ast$-algebras.
\end{proposition}

\begin{proof}
    Let $(p_v)_{v \in E^0}$ and $(s_e)_{e \in E^1}$ be the generators of $C^\ast(\Ha\Gamma)$ and let $(\hat{p}_x)_{x \in U \cup V}$ be the generators of $C^\ast(G)$ as in Definition \ref{bipartite_graph_algebras:definition}. Then we define a $\ast$-homomorphism $\varphi: C^\ast(G) \to C^\ast(\Ha\Gamma)$ by
    \begin{align*}
        \varphi(\hat{p}_x) &= p_x, & \text{if } x \in E^0, \\
        \varphi(\hat{p}_x) &= s_x, & \text{if } x \in E^1.
    \end{align*}
    It is an easy exercise to check that the relations \eqref{bga::eq:partition_relation} and \eqref{bga::eq:orthogonality_relation} are satisfied by the elements $\varphi(\hat{p}_x)$. Thus, the universal property of $C^\ast(G)$ yields that this map exists. Similarly, one defines the inverse map $\psi: C^\ast(\Ha\Gamma) \to C^\ast(G)$ using the universal property of $C^\ast(G)$.
\end{proof}

\subsection{Alternative generators of bipartite graph C*-algebras}

In the previous section, we introduced bipartite graph $C^\ast$-algebras $C^\ast(G)$ as universal $C^\ast$-algebras which are generated by projections associated to the vertices of a bipartite graph $G$. Interestingly, one can define $C^\ast(G)$ in a different way as universal $C^\ast$-algebra generated by contractions associated to the edges of $G$. This is the content of the following proposition.

\begin{proposition}
    \label{bga::proposition:alternative_definition}
    Let $G = (U, V, E)$ be a bipartite graph. Then $C^\ast(G)$ is the universal $C^\ast$-algebra generated by a family of elements $(x_e)_{e \in E}$ satisfying
    \begin{align}
        x_e^\ast x_f &= 0 & & \text{ if } e \cap f \cap U = \emptyset,  \tag{GC1} \label{bga::eq:gc1} \\
        x_e x_f^\ast &= 0 & & \text{ if } e \cap f \cap V = \emptyset,  \tag{GC2} \label{bga::eq:gc2} \\
        \left( \sum_{e \in E} x_e^\ast \right) x_f &= x_f,              \tag{GC3} \label{bga::eq:gc3} \\
        x_e \left( \sum_{f \in E} x_f^\ast \right) &= x_e,              \tag{GC4} \label{bga::eq:gc4}
    \end{align}
    for all edges $e, f \in E$, respectively. In particular, the $x_e$ are contractions\footnote{A contraction is an element $x$ with $\|x\| \leq 1$. The relation $x x^\ast x = x^2$ means that $x$ is a product of two projections, see e.g. \cite[Theorem 8]{radjavi_products_1969}.} which satisfy $x_e x_e^\ast x_e = x_e^2$.
\end{proposition}

\begin{proof}
    First of all, conditions \eqref{bga::eq:gc1} and \eqref{bga::eq:gc2} together with \eqref{bga::eq:gc3} entail for every edge $e = \{u,v\} \in E$ that
    \begin{align*}
        x_e x_e^\ast x_e 
            = x_e \left( \sum_{f \in E} x_f^\ast \right) x_e
            = x_e^2,
    \end{align*}
    since in the second expression all terms for $f \neq e$ vanish. It follows in particular that the $x_e$ are contractions, and the universal $C^\ast$-algebra $A$ generated by elements $(x_e)_{e \in E}$ satisfying \eqref{bga::eq:gc1}--\eqref{bga::eq:gc4} exists. Further, \eqref{bga::eq:gc3} and \eqref{bga::eq:gc4} imply that $A$ is unital with unit $1 = \sum_{e \in E} x_e^\ast$.

    Let us use the universal property of $C^\ast(G)$ to find a $\ast$-homomorphism 
    $$
        \varphi: C^\ast(G) \to A
    $$ 
    with
    \begin{align*}
        \varphi(p_x) = \sum_{e \in E: x \in e} x_e =: P_x \quad \text{ for all } x \in U \cup V.
    \end{align*}
    First, one checks $P_x = P_x^\ast$. Indeed, assume $x \in U$ and apply \eqref{bga::eq:gc3} and \eqref{bga::eq:gc1} to obtain
    \begin{align*}
        P_x^\ast = \sum_{e \in E: x \in e} x_e^\ast 
            = \sum_{e \in E: x \in e} x_e^\ast \left( \sum_{f \in E} x_f \right) 
            &= \sum_{e \in E} x_e^\ast \left( \sum_{f \in E: x \in f} x_f \right) \\
            &= \sum_{f \in E: x \in f} \left(\sum_{e \in E} x_e^\ast\right) x_f 
            = P_x.
    \end{align*}
    In the second step, we use \eqref{bga::eq:gc3} after taking the involution on both sides.
    If $x \in V$, one can use \eqref{bga::eq:gc4} and \eqref{bga::eq:gc2} to obtain in a similar way
    \begin{align*}
        P_x^\ast = \sum_{e \in E: x \in e} x_e^\ast 
            = \sum_{e \in E: x \in e} \left( \sum_{f \in E} x_f \right) x_e^\ast 
            &= \sum_{e \in E} \left( \sum_{f \in E: x \in f} x_f \right) x_e^\ast \\
            &= \sum_{f \in E: x \in f} x_f \left( \sum_{e \in E} x_e^\ast \right) 
            = P_x.
    \end{align*}
    Next, let us prove $P_x^2 = P_x$. Again we first assume $x \in U$. Then \eqref{bga::eq:gc1} and \eqref{bga::eq:gc3} yield
    \begin{align*}
        P_x^2 = \left( \sum_{e \in E: x \in e} x_e^\ast \right) \left( \sum_{f \in E: x \in f} x_f \right)
            = \sum_{f \in E: x \in f} \left( \sum_{e \in E} x_e^\ast \right) x_f
            = P_x.
    \end{align*}
    If $x \in V$, then one can use \eqref{bga::eq:gc2} and \eqref{bga::eq:gc4} to obtain
    \begin{align*}
        P_x^2 = \left( \sum_{e \in E: x \in e} x_e \right) \left( \sum_{f \in E: x \in f} x_f^\ast \right)
            = \sum_{f \in E: x \in f} x_f \left( \sum_{e \in E} x_e^\ast \right)
            = P_x.
    \end{align*}
    Since
    \begin{align*}
        \sum_{u \in U} P_u &= \sum_{u \in U} \sum_{e \in E: u \in e} x_e^\ast = \sum_{e \in E} x_e^\ast = 1, \text{ and } \\
        \sum_{v \in V} P_v &= \sum_{v \in V} \sum_{e \in E: v \in e} x_e^\ast = \sum_{e \in E} x_e^\ast = 1,
    \end{align*}
    the projections $(P_x)_{x \in U \cup V}$ satisfy the relations \eqref{bga::eq:partition_relation}. It remains to check the orthogonality relations \eqref{bga::eq:orthogonality_relation}. For this, let $u \in U$ and $v \in V$ be two vertices such that $\{u, v\} \not \in E$. Then one has
    \begin{align*}
        P_v P_u &= \left( \sum_{e \in E: v \in e} x_e \right) \left( \sum_{g \in E: u \in g} x_g \right) \\
            &= \left( \sum_{e \in E: v \in e} x_e \right) \left( \sum_{f \in E} x_f^\ast \right) \left( \sum_{g \in E: u \in g} x_g \right)\\
            &= \sum_{e,f,g \in E: v \in f, u \in f} x_e x_f^\ast x_g \\
            &= 0,
    \end{align*}
    where the second-last equality follows from \eqref{bga::eq:gc1} and \eqref{bga::eq:gc2}, and the last equality follows from the fact that $u$ and $v$ are not connected by an edge in $G$.  This proves \eqref{bga::eq:orthogonality_relation}. Altogether, we showed that the elements $P_x \in A$ for $x \in U \cup V$ are projections satisfying the relations \eqref{bga::eq:partition_relation} and \eqref{bga::eq:orthogonality_relation}. Thus, the universal property of $C^\ast(G)$ yields the desired $\ast$-homomorphism $\varphi: C^\ast(G) \to A$.

    Conversely, one can use the universal property of $A$ to obtain an inverse $\ast$-homomorphism $\psi: A \to C^\ast(G)$ with
    \begin{align*}
        \psi(x_e) = p_u p_v =: X_e \quad \text{ for all } e = \{u, v\} \in E,
    \end{align*}
    where we implicitly require $u \in U$ and $v \in V$.
    We need to check that the elements $X_e \in C^\ast(G)$ satisfy the relations \eqref{bga::eq:gc1}--\eqref{bga::eq:gc4}. For \eqref{bga::eq:gc1} let $e = \{u_1, v_1\}$ and $f = \{u_2, v_2\}$ with $u_1, u_2 \in U$ and $v_1, v_2 \in V$ be two edges such that $u_1 \neq u_2$. Then one has
    \begin{align*}
        X_e^\ast X_f = p_{v_1} p_{u_1} p_{u_2} p_{v_2} = 0
    \end{align*}
    since $p_{u_1} p_{u_2} = 0$ by \eqref{bga::eq:partition_relation}. The relation \eqref{bga::eq:gc2} is checked in a similar way. For \eqref{bga::eq:gc3} let $f = \{u, v\} \in E$ and observe
    \begin{align*}
        \left( \sum_{e \in E} X_e^\ast \right) X_f 
            &= \left( \sum_{e \in E} p_{v_e} p_{u_e} \right) p_u p_v \\
            &= \left( \sum_{u^\prime \in U, v^\prime \in V} p_{v^\prime} p_{u^\prime} \right) p_u p_v \\
            &= \left( \sum_{v^\prime \in V} p_{v^\prime} \right) \left( \sum_{u^\prime \in U} p_{u^\prime} \right) p_u p_v \\
            &= p_u p_v \\
            &= X_f,
    \end{align*}
    where we write $e = \{u_e, v_e\}$ for all edges $e \in E$ with $u_e \in U$ and $v_e \in V$. In the second step, we use that by \eqref{bga::eq:orthogonality_relation} all superfluous terms in the sum vanish, and in the fourth step we use \eqref{bga::eq:partition_relation}. Relation \eqref{bga::eq:gc4} is checked analogously. 

    Finally, it remains to verify that $\psi$ and $\varphi$ are inverse to each other. For $e = \{u,v\} \in E$ observe
    \begin{align*}
        \varphi(\psi(x_e)) = \varphi(p_u p_v) 
            = \varphi(p_u) \varphi(p_v)  
            &= \left( \sum_{f \in E: u \in f} x_f^\ast \right) \left( \sum_{g \in E: v \in g} x_g \right) \\
            &= \left( \sum_{f \in E: u \in f} x_f^\ast \right) x_e
            = \left(\sum_{f \in E} x_f^\ast \right) x_e
            = x_e,
    \end{align*}
    where we use \eqref{bga::eq:gc1} in the forth and fifth step, and \eqref{bga::eq:gc3} in the last step. Further, one has for $u \in U$
    \begin{align*}
        \psi(\varphi(p_u)) = \psi\left( \sum_{e \in E: u \in e} x_e \right) 
            = \sum_{e  \in E: u \in e} p_u p_{v_e}
            = p_u \left( \sum_{v^\prime \in V} p_{v^\prime} \right) 
            = p_u,
    \end{align*}
    where we use \eqref{bga::eq:orthogonality_relation} in the second-last step and \eqref{bga::eq:partition_relation} in the last step.
    One checks $\psi(\varphi(p_v))=p_v$ for $v \in V$ in a similar way. This proves that $\varphi$ and $\psi$ are inverse to each other, and hence they yield the desired isomorphism $C^\ast(G) \cong A$.
\end{proof}

\section{Projections in generic position}
\label{sec::generic_position}

Let us recall again Halmos' classical result: If $P$ and $Q$ are two projections in generic position on a Hilbert space $\mathcal{H}$, then there are contractions $C$ and $S$ on some Hilbert space $\mathcal{K}$ with $C^2 + S^2 = I$ such that up to unitary equivalence one has
\begin{align*}
    P = \begin{pmatrix}
        I & 0 \\
        0 & 0
    \end{pmatrix},
    \quad
    Q = \begin{pmatrix} 
        C^2 & CS \\
        CS & S^2
    \end{pmatrix}.
\end{align*}
Following earlier work of Vasilevski \cite{vasilevski_c-algebras_1998}, we generalize this result to a $G$-projection family for an arbitrary bipartite graph $G$. Thus, we specify what a $G$-projection family in generic position is, and we write such projections in a canonical block matrix form where the entries of the block are contractions satisfying certain relations.

Throughout this section let $G = (U, V, E)$ be a connected bipartite graph and let $(P_x)_{x \in U \cup V}$ be a $G$-projection-family on some Hilbert space $\mathcal{H}$ as defined in Definition \ref{bipartite_graph_algebras:definition}.
Further, set 
\begin{align*}
    L_x := \mathrm{im}(P_x) \subset \mathcal{H}
\end{align*}
for all $x \in U \cup V$.

\begin{definition}
    \label{genpos::def::generic_position}
    Following Halmos and Vasilevski \cite{halmos_two_1969, vasilevski_algebra_1981}, we say that the $P_x$ are in \emph{generic position} if
    \begin{align*}
        L_u \cap L_v^\perp = \{0\} = L_u^\perp \cap L_v
    \end{align*}
    holds for all edges $\{u, v\} \in E$.
\end{definition}

\begin{lemma} 
    \label{genpos::lemma::Lx_are_isomorphic}
    Assume that the $P_x$ are in generic position. Then one has 
    \begin{align*}
        L_x \cong L_y
    \end{align*}
    for all $x, y \in U \cup V$.
\end{lemma}

\begin{proof}
    We borrow the proof from \cite[Theorem 1]{halmos_two_1969}.
    As $G$ is connected it suffices to show that claim for $\{x, y\} = \{u, v\} \in E$. So let $u \in U$ and $v \in V$ be two adjacent vertices. We claim that the operator $P_u |_{L_v} : L_v \to L_u$ is injective with dense range. Indeed, for any $f \in L_v$ one has 
    \begin{align*}
        P_u f = 0
        &\implies f \in L_v \cap L_u^\perp  \\
        &\implies f = 0,
    \end{align*}
    which proves injectivity. Next, let $g \in L_u$ and assume that $g \perp P_u f$ holds for all $f \in L_v$. Then one has for all $f \in L_v$
    \begin{align*}
        0 = \langle g, P_u f \rangle = \langle P_u g, f \rangle = \langle g, f \rangle,
    \end{align*}
    and hence we see $g \in L_u \cap L_v^\perp$ which implies $g = 0$. Thus, the operator $P_u |_{L_v}$ has dense range in $L_u$. Consequently, the spaces $L_u$ and $L_v$ are isometrically isomorphic.
\end{proof}

\begin{thm} 
    \label{genpos::thm:projections_in_generic_position}
    Let $(P_x)_{x \in U \cup V}$ be a $G$-projection family on a Hilbert space $\mathcal{H}$ in generic position where $G$ is connected. Then there are operators $(C_{uv})_{u \in U, v \in V}$ on a Hilbert space $\mathcal{K}$ such that up to unitary equivalence one has
    \begin{align}
        \begin{aligned}
        P_u &= (C_{u v_1}^\ast C_{u v_2})_{v_1, v_2 \in V} \in M_{V}(\mathcal{B}(\mathcal{K})), \\
        P_v &= (\delta_{v_1 v} \delta_{v_2 v})_{v_1, v_2 \in V} \in M_{V}(\mathcal{B}(\mathcal{K})), 
        \end{aligned} \label{genpos::eq:P_in_terms_of_C}
    \end{align}
    for all $u \in U$ and $v \in V$, where $\delta_{v_1 v} = I \in B(\mathcal{K})$ if $v_1 = v$ and $\delta_{v_1 v} = 0$ otherwise, and $M_{V}(\mathcal{B}(\mathcal{K}))$ is the algebra of square matrices indexed by $V$ with entries from the bounded operators on $\mathcal{K}$.
    
    The operator $C_{uv}$ vanishes if $\{u, v\} \not \in E$, and is injective with dense range otherwise. Moreover, the operators $(C_{uv})_{u \in U, v \in V}$ satisfy
    \begin{align}
        \sum_{u \in U} C_{u v_1}^\ast C_{u v_2} &= \delta_{v_1 v_2}, \label{bga:generic_position:G-contraction_family_condition_1} \\
        \sum_{v \in V} C_{u_1 v} C_{u_2 v}^\ast &= \delta_{u_1 u_2}. \label{bga:generic_position:G-contraction_family_condition_2}
    \end{align}

    Conversely, every family of operators $(C_{uv})_{u \in U, v \in V}$ with the above properties gives rise to a $G$-projection family $(P_x)_{x \in U \cup V}$ via the formula \eqref{genpos::eq:P_in_terms_of_C}.
\end{thm}

\begin{proof}
    Assume that $(P_x)_{x \in U \cup V}$ is a $G$-projection family on the Hilbert space $\mathcal{H}$. With respect to the decomposition $\mathcal{H} = \bigoplus_{v \in V} L_v$ we can write the $P_x$ in block matrix form as 
    \begin{align*}
        P_u = (P_{v_1} P_u P_{v_2})_{v_1, v_2 \in V}, 
        \quad \text{and} \quad
        P_v = (\delta_{v_1 v} \delta_{v_2 v})_{v_1, v_2 \in V}.
    \end{align*}
    By Lemma \ref{genpos::lemma::Lx_are_isomorphic} there is a Hilbert space $\mathcal{K}$ and a family of isometric isomorphisms $U_x: L_x \to \mathcal{K}$ with $x \in U \cup V$. Let $\mathcal{U} := \mathrm{diag}((U_v)_{v \in V}) \in M_V(B(\mathcal{K}))$ and observe
    \begin{align*}
        \mathcal{U} P_u \mathcal{U}^\ast = (U_{v_1} P_{v_1} P_u U_u^\ast U_u P_u P_{v_2} U_{v_2}^\ast)_{v_1, v_2 \in V}, 
        \quad \text{and} \quad
        \mathcal{U} P_v \mathcal{U}^\ast = (\delta_{v_1 v} \delta_{v_2 v})_{v_1, v_2 \in V}.
    \end{align*}
    Setting
    \begin{align*}
        C_{uv} := U_u P_u P_v U_v^\ast \in B(\mathcal{K})
    \end{align*}
    for all $u \in U$ and $v \in V$ we immediately obtain \eqref{genpos::eq:P_in_terms_of_C}.

    Let us prove that the operators $C_{uv}$ have the desired properties. If $\{u,v\} \not \in E$ is not an edge then \eqref{bga::eq:orthogonality_relation} entails $P_u P_v = 0$ and thus $C_{uv} = 0$. Otherwise, $C_{uv}$ is injective with dense range by the proof of Lemma \ref{genpos::lemma::Lx_are_isomorphic}. Moreover, for every $u_1, u_2 \in U$ and $v_1, v_2 \in V$ one easily checks
    \begin{align*}
        \sum_{u \in U} C_{u v_1}^\ast C_{u v_2}
            &= \sum_{u \in U} U_{v_1} P_{v_1} P_u U_u^\ast U_{u} P_u P_{v_2} U_{v_2}^\ast \\
            &= \sum_{u \in U} U_{v_1} P_{v_1} P_u P_{v_2} U_{v_2}^\ast \\
            &= U_{v_1} P_{v_1} P_{v_2} U_{v_2}^\ast \\
            &= \delta_{v_1 v_2},
    \end{align*}
    using \eqref{bga::eq:partition_relation}, and 
    \begin{align*}
        \sum_{v \in V} C_{u_1 v} C_{u_2 v}^\ast 
            &= \sum_{v \in V} U_{u_1} P_{u_1} P_{v} U_{v}^\ast U_{v} P_v P_{u_2} U_{u_2}^\ast \\
            &= U_{u_1} P_{u_1} P_{u_2} U_{u_2}^\ast \\
            &= \delta_{u_1 u_2}
    \end{align*}
    using \eqref{bga::eq:partition_relation} again.

    In the other direction, assume that $(C_{uv})_{u \in U, v \in V}$ is a family of operators on some Hilbert space $\mathcal{K}$ satisfying the conditions from the statement. Let operators $P_x$ on the Hilbert space $\mathcal{K}^V$ be defined via \eqref{genpos::eq:P_in_terms_of_C}. We need to show that the $P_x$ form a $G$-projection family. Evidently, the $P_v$ with $v \in V$ are pairwise orthogonal projections adding up to the unit. On the other hand, the $P_u$ with $u \in U$ are clearly selfadjoint, and we have for every $v_1, v_2 \in V$ by (\ref{bga:generic_position:G-contraction_family_condition_2})
    \begin{align*}
        [P_{u_1} P_{u_2}]_{v_1 v_2} &= \sum_{v_3 \in V} C_{u_1 v_1}^\ast C_{u_1 v_3} C_{u_2 v_3}^\ast C_{u_2 v_2} \\
            &= C_{u_1 v_1}^\ast \left( \sum_{v_3 \in V} C_{u_1 v_3} C_{u_2 v_3}^\ast \right) C_{u_2 v_2} \\
            &= \delta_{u_1 u_2} C_{u_1 v_1}^\ast C_{u_2 v_2} \\
            &= \delta_{u_1 u_2} [P_{u_1}]_{v_1 v_2}.
    \end{align*}
    Further, using (\ref{bga:generic_position:G-contraction_family_condition_1}) one easily obtains $\sum_{u \in U} P_u = I \in M_V(B(\mathcal{K}))$. Hence, the $P_u$ with $u \in U$ form a partition of unity as well. Finally, assume that $\{u, v \}$ with $u \in U$ and $v \in V$ is not an edge in $G$. Then it is not hard to check that every entry of $P_u$ in the $v$-column or $v$-row is equal to zero, which entails $P_u \cap P_v$. Therefore, the $P_x$ form a $G$-projection family. 

    It only remains to show that the $P_x$ are in generic position. Let $\{u,v\} \in E$ be an edge. First, we show $L_u^\perp \cap L_v = \{0\}$. Assume $g = (g_{v_1})_{v_1 \in V} \in L_u^\perp \cap L_v \subset \mathcal{K}^V$. Then $g_{v_1} = 0$ for all $v_1 \neq v$ as $g \in L_v$. Thus, one checks
    \begin{align*}
        P_u g = \left( C_{u v_1}^\ast C_{u v} g_v \right)_{v_1 \in V} 
    \end{align*}
    which vanishes because of $g \in L_u^\perp$. In particular, it follows
    \begin{align*}
        \sum_{v_1 \in V} C_{u v_1} C_{u v_1}^\ast C_{u v} g_v = C_{uv}^2 g_v = 0.
    \end{align*}
    Since $C_{uv}$ is injective, we conclude $g_v = 0$ and hence $g = 0$.

    It remains to show $L_u \cap L_v^\perp = \{0\}$. Let $g = (g_{v_1}) \in L_u \cap L_v^\perp \subset \mathcal{K}^V$. Then one has $g_v = 0$ as $g \in L_v^\perp$. Further, since $g \in L_u$ we have $P_u g = g$, i.e.
    \begin{align*}
        0 = (P_u g)_v = \sum_{v_1 \in V} C_{u v}^\ast C_{u v_1} g_{v_1} = C_{u v}^\ast \left( \sum_{v_1 \in V} C_{u v_1} g_{v_1} \right).
    \end{align*}
    Consequently, for $f := \sum_{v_1 \in V} C_{u v_1} g_{v_1}$ we have $C_{uv}^\ast f = 0$. As $C_{uv}$ has dense range it follows $f = 0$. Thus, we obtain for every $v_2 \in V$
    \begin{align*}
        0 = C_{u v_2}^\ast f = \sum_{v_1 \in V} C_{u v_2}^\ast C_{u v_1} g_{v_1} = (P_u g)_{v_2} = g_{v_2}.
    \end{align*}
    We conclude $g = 0$ as desired.
\end{proof}

\section{One- and two-dimensional representations}
\label{sec::one_and_two_dimensional_representations}

Recall that the elements of the spectrum $\mathrm{Spec}(A)$ of a $C^\ast$-algebra $A$ are equivalence classes of irreducible representations of $A$ with respect to unitary equivalence. Every irreducible representation $\pi$ of $A$ induces a primitive ideal $I_\pi = \ker(\pi)$, and the set of all primitive ideals $\mathrm{Prim}(A)$ is equipped with the hull-kernel topology, i.e. the topology is given by setting the closure of a set $\mathcal{J} \subset \mathrm{Prim}(A)$ to be
\begin{align*}
    \overline{\mathcal{J}} = \left\{I \in \mathrm{Prim}(A) \mid I \supset \bigcap_{J \in \mathcal{J}} J \right\}.
\end{align*}
The topology of $\mathrm{Spec}(A)$ is then the coarsest topology that makes the map $\mathrm{Spec}(A) \to \mathrm{Prim}(A), \, \pi \mapsto I_\pi$ continuous. See \cite[Section II.6.5]{blackadar_operator_2006} for more on the spectrum of a $C^\ast$-algebra.

We say that a representation $\pi$ of $A$ is $n$-dimensional if it is a map $\pi: A \to M_n \cong B(\C^n)$. The following subspace of $\mathrm{Spec}(A)$ is of special interest to us.

\begin{definition}
    For a $C^\ast$-algebra $A$ we set
    \begin{align*}
        \mathrm{Spec}_{\leq 2}(A) := \{ \pi \in \mathrm{Spec}(A) \mid \pi \text{ is one- or two-dimensional}\}.
    \end{align*}
    This is a topological space with the subspace topology inherited from $\mathrm{Spec}(A)$.
\end{definition}

\subsection{The complete graph \texorpdfstring{$K_{2,2}$}{K22}}

To prepare a description of $\mathrm{Spec}_{\leq 2}(C^\ast(G))$ for arbitrary bipartite graphs, let us analyze the algebra $C^\ast(K_{2,2})$ of the complete bipartite graph $K_{2,2}$ in detail. We use the labels from Figure \ref{reps::fig:K22}.
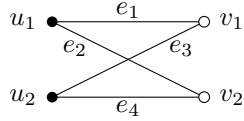
\begin{figure}[hbt]
        \begin{tikzpicture}
                \draw (-1,2) -- (1, 2);
                \draw (-1,1) -- (1, 1);
                \draw (-1,2) -- (1, 1);
                \draw (-1,1) -- (1, 2);
                \node[label=90:$e_1$] at (0, 1.8) {};
                \node[label=-90:$e_4$] at (0, 1.2) {};
                \node[label=90:$e_2$] at (-.7, 1.3) {};
                \node[label=90:$e_3$] at (.7, 1.3) {};
                \node[label=180:$u_1$, fill=black, circle, inner sep=.05cm] at (-1, 2) {};
                \node[label=180:$u_2$, fill=black, circle, inner sep=.05cm] at (-1, 1) {};
                \node[label=0:$v_1$, fill=white, draw=black, circle, inner sep=.05cm] at (1, 2) {};
                \node[label=0:$v_2$, fill=white, draw=black, circle, inner sep=.05cm] at (1, 1) {};
            \end{tikzpicture}
    \caption{The complete bipartite graph $K_{2,2}$}
    \label{reps::fig:K22}
\end{figure}
In fact, $C^\ast(K_{2,2})$ is isomorphic to the well-known universal C*-algebra $C^\ast(p,q)$ generated by two projections $p$ and $q$. This was already mentioned in Example \ref{bga::example:C_nm_as_bipartite_graph_algebra}, but we repeat this observation here with an explicit isomorphism.

\begin{proposition}
    \label{reps::prop:C*K22_is_C*pq}
    It is $C^\ast(K_{2,2}) \cong C^\ast(p,q)$. An isomorphism is given by the assignment
    \begin{align*}
        p_{u_1} &\mapsto p, &p_{v_1} &\mapsto q \\
        p_{u_2} &\mapsto 1-p, &p_{v_2} &\mapsto 1-q.
    \end{align*}
\end{proposition}

\begin{proof}
    Evidently, $\{p,1-p\}$ and $\{q, 1-q\}$ are two partitions of unity. 
    Since $K_{2,2}$ is complete, there are no particular orthogonality relations that the generators of $C^\ast(K_{2,2})$ must satisfy, i.e. \eqref{bga::eq:orthogonality_relation} is trivial. 
    Thus, the universal property of $C^\ast(K_{2,2})$ yields a $\ast$-homomorphism $\varphi: C^\ast(K_{2,2}) \to C^\ast(p,q)$ which extends the assignment from the statement of the proposition. 
    Analogously, the universal property of $C^\ast(p,q)$ yields an inverse $\ast$-homomorphism $\psi: C^\ast(p,q) \to C^\ast(K_{2,2})$ with 
    $$
    \psi: p \mapsto p_{u_1}, \; q \mapsto p_{v_1}.
    $$
\end{proof}

As discussed in the introduction, the algebra $C^\ast(p,q)$ is well-known in the literature. The following explicit description is essentially due to Pedersen \cite{pedersen_measure_1968}. Further proofs can be found e.g. in \cite{power_hankel_1982, roch_algebras_1988,raeburn_c-algebra_1989}, see also \cite{bottcher_gentle_2010}.

\begin{thm}[Pedersen 1968]
    \label{reps::thm:C*pq_as_continuous_functions}
    The universal C*-algebra $C^\ast(p,q)$ generated by two projections $p$ and $q$ is isomorphic to the algebra 
    \begin{align*}
        A := \{f \in C([0,1], M_2) \mid f(0), f(1) \text{ are diagonal} \}
    \end{align*}
    of continuous functions from $[0,1]$ to $M_2$ that assume diagonal matrices at the endpoints $0$ and $1$. An isomorphism is given by the assignment
    \begin{align*}
        p \mapsto \begin{pmatrix}
            1 & 0 \\
            0 & 0
        \end{pmatrix}_{t \in [0,1]}, 
        \quad
        q \mapsto \begin{pmatrix}
            t & \sqrt{t(1-t)} \\
            \sqrt{t(1-t)} & 1-t
        \end{pmatrix}_{t \in [0,1]}.
    \end{align*}
\end{thm}

With this explicit picture of $C^\ast(p,q)$ at hand, it is not difficult to describe the spectrum of $C^\ast(p,q)$. For that, let $X$ be the quotient of the disjoint union 
$$
    [0,1] \amalg [0,1] = \{t, t^\prime: 0 \leq t \leq 1\}
$$ 
over the equivalence relation that identifies $t$ and $t^\prime$ for $t \in (0,1)$. The underlying set of this space has the form $\{a, b, c, d\} \cup (0,1)$ and can be sketched as in Figure \ref{reps::fig:X}.
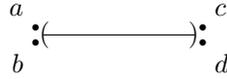
\begin{figure}[hbt]
    \begin{tikzpicture}
        \draw (0,0) -- (2,0);
        \node[label=(] at (.02,-.42) {};
        \node[label=)] at (1.98,-.42) {};
        \node[label=135:$a$, fill=black, circle, inner sep=.03cm] at (-.1, .1) {};
        \node[label=-135:$b$, fill=black, circle, inner sep=.03cm] at (-.1, -.1) {};
        \node[label=45:$c$, fill=black, circle, inner sep=.03cm] at (2.1, .1) {};
        \node[label=-45:$d$, fill=black, circle, inner sep=.03cm] at (2.1, -.1) {};
    \end{tikzpicture}
    \caption{The space $X$}
    \label{reps::fig:X}
\end{figure}
A neighborhood system of $a$ is given by $\{ \{a\} \cup (0, \frac{1}{n}): n \in \N\}$, and there are analogous neighborhood systems for $b, c$ and $d$. 
Note that the space $X$ is $T_0$ but not Hausdorff. 

\begin{corollary} 
    \label{reps::cor:X_is_Spec(K22)}
    We have
    \begin{align*}
        \mathrm{Spec}_{\leq 2}(C^\ast(K_{2,2})) = \mathrm{Spec}(C^\ast(K_{2,2})) \cong X.
    \end{align*}
    The points $a, b, c, d \in X$ correspond to the one-dimensional representations of $C^\ast(K_{2,2})$, while the points in $(0,1) \subset X$ correspond to the two-dimensional irreducible representations.
\end{corollary}

\begin{proof}
    Let $K_{2,2}$ be labeled as in Figure \ref{reps::fig:K22}. In view of the previous Theorem \ref{reps::thm:C*pq_as_continuous_functions} and Proposition \ref{reps::prop:C*K22_is_C*pq} we may assume that the generators $p_x$ of $C^\ast(K_{2,2})$ are concretely given by the functions
    \begin{align*}
        \begin{aligned}
            p_{u_1}: [0,1] &\ni t \mapsto \begin{pmatrix}
                1 & 0 \\ 0 & 0
            \end{pmatrix}, \\
            p_{u_2}: [0,1] &\ni t \mapsto \begin{pmatrix}
                0 & 0 \\ 0 & 1
            \end{pmatrix},
        \end{aligned}
        & \qquad 
        \begin{aligned}
            p_{v_1}: [0,1] &\ni t \mapsto \begin{pmatrix}
                1-t & \sqrt{t(1-t)} \\ \sqrt{t(1-t)} & t
            \end{pmatrix}, \\
            p_{v_2}: [0,1] &\ni t \mapsto \begin{pmatrix}
                t & -\sqrt{t(1-t)} \\ -\sqrt{t(1-t)} & 1- t
            \end{pmatrix}.
        \end{aligned}
    \end{align*}
    One easily checks that the two-dimensional irreducible representations of $C^\ast(K_{2,2})$ are exactly the following (up to unitary equivalence):
    \begin{align*}
        \pi_t: C^\ast(K_{2,2}) \to M_2, \; f \mapsto f(t), \quad \text{ for } t \in (0,1),
    \end{align*}
    while the one-dimensional irreducible representations are exactly (up to unitary equivalence):
    \begin{align*}
        \pi_a: C^\ast(K_{2,2}) \to \C, \; f \mapsto f(0)_{11},
        & &\pi_c: C^\ast(K_{2,2}) \to \C, \; f \mapsto f(1)_{11}, \\
        \pi_b: C^\ast(K_{2,2}) \to \C, \; f \mapsto f(0)_{22},
        & &\pi_d: C^\ast(K_{2,2}) \to \C, \; f \mapsto f(1)_{22}.
    \end{align*}
    Together, these are all irreducible representations of $C^\ast(K_{2,2})$ up to unitary equivalence. By looking at the hull-kernel topology of the primitive ideal space
    \begin{align*}
        \mathrm{Prim}(C^\ast(K_{2,2})) = \{\pi_t^{-1}(0): t \in (0,1)\} \cup \{\pi_a^{-1}(0), \pi_b^{-1}(0), \pi_c^{-1}(0), \pi_d^{-1}(0)\}
    \end{align*}
    one confirms that the map
    \begin{align*}
        \varphi: \left\{ \; \begin{aligned}
            \mathrm{Spec}(C^\ast(K_{2,2})) &\to X, \\
            \pi_t &\mapsto t, &t \in (0,1), \\
            \pi_a & \mapsto a, \\
            \pi_b & \mapsto b, \\
            \pi_c & \mapsto c, \\
            \pi_d & \mapsto d,
        \end{aligned} \right.
    \end{align*}
    is a homeomorphism of topological space. This concludes the proof.
\end{proof}

\subsection{General situation}
\label{sec:one_two_dimensional_representations}

In what follows we investigate the connection between an arbitrary graph $G$ and $\mathrm{Spec}_{\leq 2}(C^\ast(G))$. It turns out that the one-dimensional irreducible representations correspond to the edges of the graph $G$, while the two-dimensional ones correspond to subgraphs that are isomorphic to $K_{2,2}$. This allows for a complete description of $\mathrm{Spec}_{\leq 2}(C^\ast(G))$ from the combinatorial structure of $G$.

Throughout this section let $G = (U, V, E)$ be a bipartite graph and let
\begin{align*}
    \mathcal{G} = \{H \subset G \mid H \cong K_{2,2}\}
\end{align*}
be the collection of all subgraphs of $G$ that are isomorphic to $K_{2,2}$.

\begin{lemma} 
    \label{reps::lemma:representations_vs_edges_and_subgraphs}
    The following statements hold:
    \begin{enumerate}[label=(\alph*)]
        \item Let $\pi$ be a one-dimensional irreducible representation of $C^\ast(G)$. Then there is an edge $\{u_0, v_0\} \in E$ such that for all vertices $x \in U \cup V$ one has
        \begin{align}
            \label{reps::eq:one_dimensional_irreducible_representation_pi_e}
            \pi(p_x) = \begin{cases}
                1, &\text{if } x \in  \{u_0, v_0\}, \\
                0, &\text{otherwise}.
            \end{cases}
        \end{align}
        Moreover, for every edge $e \in E$ there is exactly one such irreducible representation, which we denote by $\pi_e$.

        \item Let $\sigma$ be a two-dimensional irreducible representation of $C^\ast(G)$. Then there are vertices $u_1, u_2, v_1, v_2 \in U \cup V$ with $G(u_1, u_2, v_1, v_2) \cong K_{2, 2}$ such that up to a unitary transformation one has
        \begin{align} 
            \label{reps::eq:two_dimensional_irreducible_representation_sigma}
            \begin{aligned}
            &\begin{aligned}
                \sigma(p_{u_1}) &= \begin{pmatrix}
                1 & 0 \\ 0 & 0
                \end{pmatrix}, \\
                \sigma(p_{u_2}) &= \begin{pmatrix}
                0 & 0 \\ 0 & 1
                \end{pmatrix},
            \end{aligned}
            \quad \begin{aligned}
                \sigma(p_{v_1}) &= \begin{pmatrix}
                1-t & \sqrt{t(1-t)} \\ \sqrt{t(1-t)} & t
                \end{pmatrix}, \\
                \sigma(p_{v_2}) &= \begin{pmatrix}
                t & -\sqrt{t(1-t)} \\ -\sqrt{t(1-t)} & 1-t
                \end{pmatrix},
            \end{aligned}  \\
            & \sigma(p_x) = 0 \quad \text{ for all } x \in (U \cup V) \setminus \{u_1, u_2, v_1, v_2\}
            \end{aligned}
        \end{align}
        for some $t \in (0,1)$.
        
        Conversely, for every such vertices and some $t \in (0,1)$ there exists exactly one such irreducible representation, which we denote by $\sigma_{H, t}$ for $H = G(u_1, u_2, v_1, v_2)$.
    \end{enumerate}
    In particular, one has the identity of sets
    \begin{align*}
        \mathrm{Spec}_{\leq 2}(C^\ast(G)) = \{ \pi_e \mid e \in E \} \cup \{ \sigma_{H, t} \mid H \in \mathcal{G}, t \in (0,1) \}.
    \end{align*}
\end{lemma}

\begin{proof}
    Ad (a): 
    Clearly,
    \begin{align*}
        1 = \pi(1) = \sigma\left(\sum_{u \in U} p_u\right) = \pi\left(\sum_{v \in V} p_v\right),
    \end{align*}
    and $\pi(p_u), \pi(p_v) \in \{0,1\}$ for all $u \in U, v \in V$. Hence, there exist $u_0 \in U$ and $v_0 \in V$ with $1 = \pi(p_{u_0}) = \pi(p_{v_0})$. Assume $u_0 \not \sim v_0$. Using the relations \eqref{bga::eq:partition_relation} and \eqref{bga::eq:orthogonality_relation} one obtains
    \begin{align*}
        1 = \pi(p_{u_0}) = \pi\left(p_{u_0} \left( \sum_{v \in V} p_v \right) \right) &= \pi\left(p_{u_0} \left( \sum_{v \in \mathcal{N}(u_0)} p_v \right) \right) \\
            &= \pi(p_{u_0}) \sum_{v \in \mathcal{N}(u_0)} \pi(p_v) = 0.
    \end{align*}
    This is a contraction. Thus, we must have $u_0 \sim v_0$.

    For the last statement, assume that vertices $u_0 \sim v_0$ are given. Then the universal property of $C^\ast(G)$ yields a (unique) $\ast$-homomorphism $\pi: C^\ast(G) \to \C$ with \eqref{reps::eq:one_dimensional_irreducible_representation_pi_e}. Indeed, it is not hard to check that the elements $\pi(p_u)$ and $\pi(p_v)$ satisfy the relations \eqref{bga::eq:partition_relation} and \eqref{bga::eq:orthogonality_relation} required from $C^\ast(G)$. Evidently, $\pi$ is irreducible.

    Ad (b):
    As
    \begin{align*}
        \begin{pmatrix}
            1 & 0 \\ 0 & 1
        \end{pmatrix}
        = \sum_{u \in U} \sigma(p_u) = \sum_{v \in V} \sigma(p_v)
    \end{align*}
    and every $\sigma(p_x)$ for $x \in U \sqcup V$ is a projection,
    there are vertices $u_1, u_2, v_1$ and $v_2$ with 
    \begin{align*}
        \begin{pmatrix}
            1 & 0 \\ 0 & 1
        \end{pmatrix}
        = \sigma(p_{u_1}) + \sigma(p_{u_2}) = \sigma(p_{v_1}) + \sigma(p_{v_2})
    \end{align*}
    and 
    $$ \sigma(p_x) = \begin{pmatrix}
        0 &0 \\ 0 &0
    \end{pmatrix} $$
    for all $x \in (U \sqcup V) \setminus \{u_1, u_2, v_1, v_2\}$. 
    For a contradiction assume
    \begin{align*}
        \sigma(p_{u_1}) = \begin{pmatrix}
            0 & 0 \\ 0 & 0
        \end{pmatrix}
        \quad \text{ and } \quad 
        \sigma(p_{u_2}) = \begin{pmatrix}
            1 & 0 \\ 0 & 1
        \end{pmatrix}.
    \end{align*}
    Then, $\mathrm{im}(\sigma)$ is the closed span of $\sigma(p_{v_1})$ and $\sigma(p_{v_2})$ which are two orthogonal projections that sum up to the unit. This would mean, evidently, that $\sigma$ is not irreducible -- contradicting the assumption. Thus, all four matrices $\sigma(p_{u_1}), \sigma(p_{u_2}), \sigma(p_{v_1})$ and $\sigma(p_{v_2})$ are projections onto one-dimensional subspaces of $\C^2$.
    By applying a unitary transformation if necessary, we can assume $\sigma(p_{u_1}) = E_{11}$ and $\sigma_{p_{u_2}} = E_{22}$. Then, $\sigma(p_{v_1})$ is the projection onto some non-vanishing vector $a e_1 + b e_2 \in \C^2$, and by applying a diagonal unitary transformation we can assume $a, b \in [0, \infty)$. After normalizing this vector, we have that $\sigma(p_{v_1})$ is the projection onto $t e_1 + (1-t) e_2$ for some $t \in [0,1]$ with $t^2 + (1-t)^2 = 1$. Consequently, we get
    \begin{align*}
        \sigma(p_{v_1}) = \begin{pmatrix}
            1-t & \sqrt{t(1-t)} \\ \sqrt{t(1-t)} & t
        \end{pmatrix}
        \quad \text{and} \quad 
        \sigma(p_{v_2}) = \begin{pmatrix}
            t & -\sqrt{t(1-t)} \\ -\sqrt{t(1-t)} & 1-t
        \end{pmatrix}.
    \end{align*}
    Further, one observes immediately $t \not \in \{0,1\}$ since otherwise $\sigma$ would not be irreducible. 
    It follows, in particular, $\sigma(p_{u_i}) \sigma(p_{v_j}) \neq 0$ for $i, j \leq 2$. Similarly as before, \eqref{bga::eq:partition_relation} and \eqref{bga::eq:orthogonality_relation} then entail $u_i \sim v_j$ for all $i,j$ which means that $G(u_1, u_2, v_1, v_2) \cong K_{2,2}$.

    For the final statement, one readily checks that the elements $\sigma(p_{u})$ and $\sigma(p_v)$ given by \eqref{reps::eq:two_dimensional_irreducible_representation_sigma} satisfy the relations \eqref{bga::eq:partition_relation} and \eqref{bga::eq:orthogonality_relation}. Therefore, the universal property of $C^\ast(G)$ yields a (unique) $\ast$-homomorphism $\sigma: C^\ast(G) \to M_2$ satisfying \eqref{reps::eq:two_dimensional_irreducible_representation_sigma}. Evidently, $\mathrm{im}(\sigma) = M_2$, and therefore the representation $\sigma$ is irreducible.
\end{proof}

\begin{example}
    \label{reps::ex:edges_and_Spec(C*(K22))}
    Let us again take a look at the complete bipartite graph $K_{2,2}$ which was discussed in the previous section. There we saw that the spectrum of $C^\ast(K_{2,2})$ is isomorphic to the space $X = \{a, b, c, d\} \sqcup (0,1)$, where the points $a, b, c, d$ correspond to one-dimensional irreducible representations. A sketch of this space can be seen in Figure \ref{reps::fig:X}. The points $a$ and $b$ ($c$ and $d$) have no disjoint open neighborhoods, while all other pairs of distinct points from $\{a, b, c, d\}$ have disjoint open neighborhoods. By the previous lemma, these points correspond to the four edges of $K_{2,2}$. In the proof of Corollary \ref{reps::cor:X_is_Spec(K22)} the irreducible representations $\pi_a, \pi_b, \pi_c$ and $\pi_d$ corresponding to the points $a, b, c, d \in X$ were described explicitly. Let us now check which edges correspond to which points. In the terminology of Corollary \ref{reps::cor:X_is_Spec(K22)}, we have
    \begin{align*}
        C^\ast(K_{2,2}) = \{f \in C([0,1], M_2) \mid f(0), f(1) \text{ are diagonal matrices}\}, \\
        \begin{aligned}
            p_{u_1} &= \begin{pmatrix}
                1 & 0 \\ 0 & 0
            \end{pmatrix}_{t \in [0,1]},
            &p_{v_1} &= \begin{pmatrix}
                1-t & \sqrt{t(1-t)} \\ \sqrt{t(1-t)} & t
            \end{pmatrix}_{t \in [0,1]}, \\
            p_{u_2} &= \begin{pmatrix}
                0 & 0 \\ 0 & 1
            \end{pmatrix}_{t \in [0,1]},
            &p_{v_2} &= \begin{pmatrix}
                t & -\sqrt{t(1-t)} \\ -\sqrt{t(1-t)} & 1-t
            \end{pmatrix}_{t \in [0,1]},
        \end{aligned}
    \end{align*}
    and the representations $\pi_a, \pi_b, \pi_c$ and $\pi_d$ are given by
    \begin{align*}
        \pi_a: C^\ast(K_{2,2}) \to \C, \; f &\mapsto f(0)_{11},
        &\pi_c: C^\ast(K_{2,2}) \to \C, \; f &\mapsto f(1)_{11}, \\
        \pi_b: C^\ast(K_{2,2}) \to \C, \; f &\mapsto f(0)_{22},
        &\pi_d: C^\ast(K_{2,2}) \to \C, \; f &\mapsto f(1)_{22}.
    \end{align*}
    It follows that $\pi_a(u_1) = \pi_a(v_1) = 1$ and $\pi_a(u_2) = \pi_a(v_2) = 0$. Using the edge labels from Figure \ref{reps::fig:K22} we have $\{u_1, v_1\} = e_1$, and we see $\pi_a = \pi_{e_1}$, where $\pi_{e_1}$ is the one-dimensional irreducible representation corresponding to the edge $e_1$ from the previous lemma. Checking the other representations in the same way, one arrives at the following correspondences:
    \begin{align*}
        a & &\leftrightarrow & &\pi_a = \pi_{e_1} & &\leftrightarrow& & e_1 = \{u_1, v_1\}, \\
        b & &\leftrightarrow & &\pi_b = \pi_{e_4} & &\leftrightarrow& & e_4 = \{u_2, v_2\}, \\
        c & &\leftrightarrow & &\pi_c = \pi_{e_2} & &\leftrightarrow& & e_2 = \{u_1, v_2\}, \\
        d & &\leftrightarrow & &\pi_d = \pi_{e_3} & &\leftrightarrow& & e_3 = \{u_2, v_1\}.
    \end{align*}
    Interestingly, the pairs $\{a,b\}$ and $\{c,d\}$ of points without disjoint open neighborhoods correspond to the pairs of edges $\{e_1, e_4\}$ and $\{e_2, e_3\}$ which are the only two pairs of non-adjacent edges in $K_{2,2}$. Thus, the spectrum $X$ of $C^\ast(K_{2,2})$ recalls which edges of $K_{2,2}$ are adjacent and which are not. This is an important feature which we will use later on. Therefore, we summarize the results of this example in the following lemma.
\end{example}

\begin{lemma}
    \label{reps::lemma:adjacency_of_K22edges_and_Spec(C*(K22))}
    Let the complete graph $K_{2,2}$ have edge labels as in Figure \ref{reps::fig:K22}. In $\mathrm{Spec}_{\leq 2}(C^\ast(K_{2,2}))$ the one-dimensional irreducible representations $\pi_{e_1}, \pi_{e_2}, \pi_{e_3}$ and $\pi_{e_4}$ from Lemma \ref{reps::lemma:representations_vs_edges_and_subgraphs} have the following property: Whenever $e_i$ and $e_j$ are two distinct edges, then $\pi_{e_i}$ and $\pi_{e_j}$ have disjoint open neighborhoods if and only if $e_i$ and $e_j$ are adjacent.
\end{lemma}

\begin{proof}
    See Example \ref{reps::ex:edges_and_Spec(C*(K22))}. The proof follows from the fact that the points $a, b, c$ and $d$ in $X$ correspond to the edges $e_1, e_2, e_3$ and $e_4$ in such a way that $\{a,b\}$ and $\{c,d\}$ correspond to non-adjacent edges, while all other pairs of distinct points from $\{a, b, c, d\}$ correspond to adjacent edges.
\end{proof}

In Lemma \ref{reps::lemma:Spec<=2(C*G)} we obtained a description of $\mathrm{Spec}_{\leq 2}(C^\ast(G))$ as a set. In the next lemma, we describe its topology using the space $X = \{a, b, c, d\} \sqcup (0,1)$ from Corollary \ref{reps::cor:X_is_Spec(K22)} as a building block.

\begin{lemma}
    \label{reps::lemma:Spec<=2(C*G)}
    Recall that $\mathcal{G} = \{H \subset G \mid H \subset G, H \cong K_{2,2}\}$ is the collection of all subgraphs of $G$ that are isomorphic to $K_{2,2}$, and recall from Lemma \ref{reps::lemma:representations_vs_edges_and_subgraphs} that
    \begin{align*}
        \mathrm{Spec}_{\leq 2}(C^\ast(G)) = \{\pi_e \mid e \in E\} \sqcup \{\sigma_{H,t} \mid H \in \mathcal{G}, 0 < t < 1\}.
    \end{align*}
    For every $H \in \mathcal{G}$ the set
    \begin{align*}
        X_H := \{\pi_{e_{H,1}}, \pi_{e_{H,2}}, \pi_{e_{H,3}}, \pi_{e_{H,4}}\} \sqcup \{\sigma_{H,t} \mid 0 < t < 1\} \subset \mathrm{Spec}_{\leq 2}(C^\ast(G))
    \end{align*}
    is a closed subset that is isomorphic to the space $X$ from Corollary \ref{reps::cor:X_is_Spec(K22)}, where $e_{H,1}, e_{H,2}, e_{H,3}$ and $e_{H,4}$ are the edges of $H$.
    Further, there is a homeomorphism
    \begin{align*}
        \bigcup_{H \in \mathcal{G}} \{\sigma_{H,t} \mid 0 < t < 1\} \cong \coprod_{H \in \mathcal{G}} (0,1),
    \end{align*}
    which maps every set $\{\sigma_{H,t} \mid 0 < t < 1\}$ onto a different copy of $(0,1)$.

    Finally, the following is true:
    \begin{enumerate}[label=(\alph*)]
        \item $\pi_e$ is clopen, i.e. $\{\pi_e\}$ is closed and open, in $\mathrm{Spec}_{\leq 2}(C^\ast(G))$ if $e$ is not contained in any $H \in \mathcal{G}$,
        \item if $e_1, e_2 \in E$ are contained in some $H \in \mathcal{G}$ then $\pi_{e_1}$ and $\pi_{e_2}$ have disjoint open neighborhoods in $\mathrm{Spec}_{\leq 2}(C^\ast(G))$ if and only if $e_1$ and $e_2$ are adjacent,
        \item the points $\sigma_{H,t}$ have an open neighborhood in $\mathrm{Spec}_{\leq 2}(C^\ast(G))$ that is homeomorphic to $(0,1)$ for every $H \in \mathcal{G}$ and $0 < t < 1$.
    \end{enumerate}
\end{lemma}

\begin{proof}
    Let $H \in \mathcal{G}$ be fixed. One checks that the set $X_H$ consists of all $1$- and $2$-dimensional irreducible representations of $C^\ast(G)$ that vanish on the projections $p_x$ for $x \not \in H$. Let $I_H \subset C^\ast(G)$ be the (closed) ideal generated by these projections $p_x$ with $x \not\in H$. The irreducible representations that vanish on $I_H$ form a closed subset of $\mathrm{Spec}(C^\ast(G))$ and are in a 1-1 correspondence with the irreducible representations of the quotient $C^\ast(G) / I_H$ via
    \begin{align*}
        \mathrm{Spec}(C^\ast(G) / I_H) \to \{\pi \in \mathrm{Spec}(C^\ast(G)) \mid \pi(p_x) = 0 \text{ for all } x \not \in H\}, \quad \pi \mapsto \pi \circ \iota,
    \end{align*}
    where $\iota: C^\ast(G) \to C^\ast(G) / I_H$ is the canonical quotient map. Moreover, the map $\pi \to \pi \circ \iota$ is a homeomorphism of topological spaces, see e.g. \cite[II.6.1.3, II.6.5.13]{blackadar_operator_2006}. Since, the map also preserves the dimension of the representations, we get a homeomorphism
    \begin{align*}
        \mathrm{Spec}_{\leq 2}(C^\ast(G) / I_H) \to \{ \pi \in \mathrm{Spec}_{\leq 2}(C^\ast(G)) \mid \pi(p_x) = 0 \text{ for all } x \not \in H\} = X_H,
    \end{align*}
    where $X_H \subset \mathrm{Spec}_{\leq 2}(C^\ast(G))$ is a closed subset. 
    By Proposition \ref{bga::proposition:algebra_of_subgraphs} the algebra $C^\ast(G) / I_H$ is isomorphic to $C^\ast(H) \cong C^\ast(K_{2,2})$. For this algebra we computed $\mathrm{Spec}_{\leq 2}(C^\ast(K_{2,2})) \cong X$ in Corollary \ref{reps::cor:X_is_Spec(K22)}.

    To prove the second statement, let us note that the sets $\{\sigma_{H,t} \mid 0 < t < 1\}$ are pairwise disjoint for different $H \in \mathcal{G}$ and isomorphic to $(0,1)$ because of $X_H \cong X$. For
    \begin{align*}
        \bigcup_{H \in \mathcal{G}} \{\sigma_{H,t} \mid 0 < t < 1\} \cong \coprod_{H \in \mathcal{G}} (0,1),
    \end{align*}
    it suffices to show that $\{\sigma_{H,t} \mid 0 < t < 1\}$ is closed in the subspace topology of $\bigcup_{H \in \mathcal{G}} \{\sigma_{H,t} \mid 0 < t < 1\}$. However, this follows immediately from the fact that $X_H \subset \mathrm{Spec}_{\leq 2}(C^\ast(G))$ is closed. Thus, the map
    \begin{align*}
        \bigcup_{H \in \mathcal{G}} \{\sigma_{H,t} \mid 0 < t < 1\} \to \coprod_{H \in \mathcal{G}} (0,1), \;
        \sigma_{H, t} \mapsto t_H
    \end{align*}
    is a homeomorphism, where for every $t \in (0,1)$ we denote by $t_H$ the $H$-labeled copy of $t$ in the disjoint union $\coprod_{H \in \mathcal{G}} (0,1)$. 

    It remains to prove (a)--(c). For (a), let $e \in E$ be an edge that is not contained in any $H \in \mathcal{G}$, and let $G^\prime$ be the graph obtained from $G$ by deleting the edge $e$. From Corollary \ref{bga::lemma:loose_edge_gives_direct_summand_C} one gets
    \begin{align*}
        C^\ast(G) \cong C^\ast(G^\prime) \oplus \C,
    \end{align*}
    and $\pi_e$ corresponds to the representation
    \begin{align*}
        C^\ast(G^\prime) \oplus \C \ni (x,y) \mapsto y \in \C.
    \end{align*}
    It follows immediately that $\{\pi_e\} \subset \mathrm{Spec}_{\leq 2}(C^\ast(G))$ is clopen.

    For (b) let $H = G(e_1, e_2, e_3, e_4) \in \mathcal{G}$. After identifying $C^\ast(H)$ with $C^\ast(G) / I_H$ the above yields a homeomorphism
    \begin{align*}
        \mathrm{Spec}_{\leq 2}(C^\ast(H)) &\to X_H \subset \mathrm{Spec}_{\leq 2}(C^\ast(G)), \\
        \pi &\mapsto \pi \circ \iota,
    \end{align*}
    where the quotient map $\iota: C^\ast(G) \to C^\ast(H)$ is given by
    \begin{align*}
        \iota(p_x) = 
        \begin{cases}
            \hat{p}_x & \text{if } x \in H, \\
            0 & \text{if } x \not\in H.
        \end{cases}
    \end{align*}
    To avoid confusion we denote the generators of $C^\ast(H)$ by $\hat{p}_x$ for $x \in H$. Thus, for the edges $e_i$ the homeomorphism maps
    \begin{align*}
        \hat{\pi}_{e_i} \mapsto \pi_{e_i},
    \end{align*}
    where $\hat{\pi}_{e_i}$ is the one-dimensional irreducible representation of $C^\ast(H)$ corresponding to the edge $e_i$ according to Lemma \ref{reps::lemma:representations_vs_edges_and_subgraphs}. As $H \cong K_{2,2}$, Lemma \ref{reps::lemma:adjacency_of_K22edges_and_Spec(C*(K22))} tells us that $e_1$ and $e_2$ are adjacent in $H$ if and only if $\hat{\pi}_{e_1}$ and $\hat{\pi}_{e_2}$ have disjoint open neighborhoods in $\mathrm{Spec}_{\leq 2}(C^\ast(H))$. It follows that $\pi_{e_1}$ and $\pi_{e_2}$ have disjoint open neighborhoods in $\mathrm{Spec}_{\leq 2}(C^\ast(G))$ if and only if $e_1$ and $e_2$ are adjacent in $H$. Evidently, this is equivalent to the fact that $e_1$ and $e_2$ are adjacent in $G$, which concludes the proof of (b).

    Finally, for (c) let $H \in \mathcal{G}$ and $0 < t < 1$. It is not hard to check that $\{\sigma_{H,t} \mid 0 < t < 1\}$ is an open neighborhood of $\sigma_{H,t}$ in $\mathrm{Spec}_{\leq 2}(C^\ast(G))$ which is homeomorphic to $(0,1)$.
\end{proof}

In the next lemma we obtain a bijection between the edges of two graphs $G$ and $G^\prime$ starting from a homeomorphism between $\mathrm{Spec}_{\leq 2}(C^\ast(G))$ and $\mathrm{Spec}_{\leq 2}(C^\ast(G^\prime))$. This will be used in the next section to obtain a $\ast$-isomorphism between $C^\ast(G)$ and $C^\ast(G^\prime)$.

\begin{lemma}
    \label{reps::lemma:edge_bijection_f_from_homeomorphism_of_spectrum}
    Let $G$ and $G^\prime$ be two bipartite graphs, and assume that
    \begin{align*}
        \Phi: \mathrm{Spec}_{\leq 2}(C^\ast(G)) \to \mathrm{Spec}_{\leq 2}(C^\ast(G^\prime))
    \end{align*}
    is a homeomorphism. 
    There is a unique bijection $f: E \to E^\prime$ given by $\Phi(\pi_e) = \pi_{f(e)}$ for all edges $e \in E$.
    This map satisfies for all edges $e_1, e_2, e_3, e_4 \in E$ 
    \begin{align}
        \label{reps::eq:f_preserves_K22_subgraphs}
        G(e_1, e_2, e_3, e_4) \cong K_{2,2}
        \quad\Leftrightarrow\quad
        G^\prime(f(e_1), f(e_2), f(e_3), f(e_4)) \cong K_{2,2}.
    \end{align}
    Further, whenever $G(e_1, e_2, e_3, e_4) \cong K_{2,2}$, then one has for distinct $i, j \leq 4$ 
    \begin{align}
        \label{reps::eq:f_preserves_adjacency}
        e_i \text{ and } e_j \text{ are adjacent} \quad\Leftrightarrow\quad f(e_i) \text{ and } f(e_j) \text{ are adjacent}.
    \end{align}
\end{lemma}

\begin{proof}
    Due to Lemma \ref{reps::lemma:Spec<=2(C*G)}(a)--(c) the representations $\pi_e$ for $e \in E$ and $\sigma_{H,t}$ for $H \in \mathcal{G}$ and $0 < t < 1$ can be distinguished by their topological properties. Therefore, the map $\Phi$ satisfies
    \begin{align*}
        \Phi\left( \{ \pi_e \mid e \in E\} \right) &= \{ \pi_{e^\prime} \mid e^\prime \in E^\prime\}, \\
        \Phi\left( \{ \sigma_{H,t} \mid H \in \mathcal{G}, 0 < t < 1\} \right) &= \{ \sigma_{H^\prime, t} \mid H^\prime \in \mathcal{G}^\prime, 0 < t < 1\}.
    \end{align*}
    It follows that the map $f$ is a well-defined bijection from $E$ to $E^\prime$. It remains to prove that $f$ satisfies \eqref{reps::eq:f_preserves_K22_subgraphs} and \eqref{reps::eq:f_preserves_adjacency}.

    For that, let $X_H \subset \mathrm{Spec}_{\leq 2}(C^\ast(G))$ be the closed subset corresponding to a subgraph $H \in \mathcal{G}$ as in Lemma \ref{reps::lemma:Spec<=2(C*G)}. We claim that for every $H \in \mathcal{G}$ there is an $H^\prime \in \mathcal{G}^\prime$ such that $\Phi(X_H) = X_{H^\prime}$, and conversely for every $H^\prime \in \mathcal{G}$ there is an $H \in \mathcal{G}$ with this property.
    
    Indeed, $X_H$ is the closure of $\{\sigma_{H,t} \mid 0 < t < 1\}$ and by the previous lemma we have 
    \begin{align*}
        \bigcup_{H \in \mathcal{G}} \{\sigma_{H,t} \mid 0 < t < 1\} & \cong \coprod_{H \in \mathcal{G}} (0,1), \\
        \bigcup_{H^\prime \in \mathcal{G}^\prime} \{\sigma_{H^\prime, t} \mid 0 < t <1\} & \cong \coprod_{H^\prime \in \mathcal{G}^\prime} (0,1),
    \end{align*}
    where the homeomorphism maps different sets $\{\sigma_{H,t} \mid 0 < t < 1\}$ for $H \in \mathcal{G}$ (or $\{\sigma_{H^\prime, t} \mid 0 < t < 1\}$ for $H^\prime \in \mathcal{G}^\prime$, resp.) onto different copies of $(0,1)$.
    As a homeomorphism $\Phi$ preserves connected components, and it follows that for every $H \in \mathcal{G}$ there exists a unique $H^\prime \in \mathcal{G}^\prime$ such that $\Phi\left( \{ \sigma_{H,t} \mid 0 < t < 1\} \right) = \{ \sigma_{H^\prime, t} \mid 0 < t < 1\}$. By taking the closure we obtain the claim. The same argument applied on $\Phi^{-1}$ shows that for every $H^\prime \in \mathcal{G}^\prime$ there is a unique $H \in \mathcal{G}$ such that $\Phi(X_H) = X_{H^\prime}$.

    Now, assume that $H = G(e_1, e_2, e_3, e_4) \in \mathcal{G}$. Then one has
    \begin{align*}
        \Phi(X_H) &= X_{H^\prime}
    \end{align*}
    for some $H^\prime = G^\prime(e_1^\prime, e_2^\prime, e_3^\prime, e_4^\prime) \in \mathcal{G}^\prime$. It follows
    \begin{align*}
        \Phi&\left( \{ \pi_{e_1}, \pi_{e_2}, \pi_{e_3}, \pi_{e_4} \} \sqcup \{ \sigma_{H,t} \mid 0 < t < 1\} \right) \\
            &= \{ \pi_{e_1^\prime}, \pi_{e_2^\prime}, \pi_{e_3^\prime}, \pi_{e_4^\prime} \} \sqcup \{ \sigma_{H^\prime, t} \mid 0 < t < 1\}
    \end{align*}
    and thus
    \begin{align*}
        \{f(e_1), f(e_2), f(e_3), f(e_4)\} = \{e_1^\prime, e_2^\prime, e_3^\prime, e_4^\prime\}.
    \end{align*}
    This proves the forward implication in \eqref{reps::eq:f_preserves_K22_subgraphs}. The converse implication is proved in the same way using that for every $H^\prime \in \mathcal{G}^\prime$ there is a unique $H \in \mathcal{G}$ such that $\Phi(X_H) = X_{H^\prime}$.

    Finally, let $e_1, e_2 \in E$ be two distinct edges in some $H \in \mathcal{G}$. Then $e_1$ and $e_2$ are adjacent if and only if $\pi_{e_1}$ and $\pi_{e_2}$ have disjoint open neighborhoods in $\mathrm{Spec}_{\leq 2}(C^\ast(G))$ by Lemma \ref{reps::lemma:Spec<=2(C*G)}(b). Similarly, $f(e_1)$ and $f(e_2)$ are adjacent if and only if $\pi_{f(e_1)}$ and $\pi_{f(e_2)}$ have disjoint open neighborhoods in $\mathrm{Spec}_{\leq 2}(C^\ast(G^\prime))$. Since $\Phi$ is a homeomorphism we have that $\pi_{f(e_1)} = \Phi(\pi_{e_1})$ and $\pi_{f(e_2)} = \Phi(\pi_{e_2})$ have disjoint open neighborhoods in $\mathrm{Spec}_{\leq 2}(C^\ast(G^\prime))$ if and only if $\pi_{e_1}$ and $\pi_{e_2}$ have disjoint open neighborhoods in $\mathrm{Spec}_{\leq 2}(C^\ast(G))$. This proves \eqref{reps::eq:f_preserves_adjacency}.
\end{proof}

\section{Classification of bipartite graph C*-algebras}
\label{sec::classification}

Finally, everything is ready to show that $\mathrm{Spec}_{\leq 2}(C^\ast(G))$ determines $C^\ast(G)$ up to isomorphism. Throughout this section let $G = (U, V, E)$ and $G^\prime = (U^\prime, V^\prime, E^\prime)$ be two fixed bipartite graphs together with a homeomorphism
$$
    \Phi: \mathrm{Spec}_{\leq 2}(C^\ast(G)) \to \mathrm{Spec}_{\leq 2}(C^\ast(G^\prime)).
$$
As seen in Lemma \ref{reps::lemma:edge_bijection_f_from_homeomorphism_of_spectrum}, the homeomorphism $\Phi$ yields a bijection $f: E \to E^\prime$ between the respective edge sets of $G$ and $G^\prime$ such that for any edges $e_1, e_2, e_3, e_4 \in E$ we have 
\begin{align}
    \tag{\ref{reps::eq:f_preserves_K22_subgraphs}}
    G(e_1, e_2, e_3, e_4) \cong K_{2,2}
    \quad\Leftrightarrow\quad
    G^\prime(f(e_1), f(e_2), f(e_3), f(e_4)) \cong K_{2,2}.
\end{align}
Additionally, if $G(e_1, e_2, e_3, e_4) \cong K_{2,2}$, then for $i, j \leq 4$ one has
\begin{align}
    \tag{\ref{reps::eq:f_preserves_adjacency}}
    e_i \text{ and } e_j \text{ are adjacent} \quad\Leftrightarrow\quad f(e_i) \text{ and } f(e_j) \text{ are adjacent}.
\end{align}

In this section, we prove that $f$ gives rise to a $\ast$-isomorphism $\varphi: C^\ast(G) \to C^\ast(G^\prime)$. First, we show that any map $f: E \to E^\prime$ satisfying \eqref{reps::eq:f_preserves_K22_subgraphs} and \eqref{reps::eq:f_preserves_adjacency} induces a $\ast$-homormophism $\varphi_f: C^\ast(G) \to C^\ast(G^\prime)$.

\begin{proposition}
    \label{class::prop:construct_isomorphism_phi_from_f}
    Let $G, G^\prime$ be two bipartite graphs and let $f: E \to E^\prime$ be a bijection between their edges which satisfies \eqref{reps::eq:f_preserves_K22_subgraphs} and \eqref{reps::eq:f_preserves_adjacency}. Then there is a (unique) $\ast$-homomorphism $\varphi_f: C^\ast(G) \to C^\ast(G^\prime)$ such that for every vertex $x \in U \cup V$ of $G$ one has
    \begin{align*}
        \varphi_f(p_x) = \sum_{\{u^\prime, v^\prime\} \in I(x)} \hat{p}_{u^\prime} \hat{p}_{v^\prime} =: P_x \in C^\ast(G^\prime),
    \end{align*}
    where 
    $$
        I(x) := \{f(e) \in E^\prime: x \in e\}
    $$
    and $\{u^\prime, v^\prime\} \in E^\prime$ implicitly requires that $u^\prime \in U^\prime$ and $v^\prime \in V^\prime$. Further, we denote the generators of $C^\ast(G^\prime)$ by $\hat{p}_y$ with $y \in U^\prime \cup V^\prime$ to avoid confusion with the generators $p_x$ with $x \in U \cup V$ of $C^\ast(G)$.
\end{proposition}

To prove this, we need to show that the $P_x$ are projections which satisfy the relations \eqref{bga::eq:partition_relation} and \eqref{bga::eq:orthogonality_relation} required for the generators of $C^\ast(G)$, i.e.
\begin{itemize}
    \item every $P_x$ with $x \in U \cup V$ is a projection  in $C^\ast(G^\prime)$,
    \item it is $\sum_{u \in U} P_u = 1 = \sum_{v \in V} P_v$, and
    \item we have $P_x \perp P_y$ unless $x \sim y$.
\end{itemize}
We will prove these properties in the next lemmas separately. Throughout the next lemmas, $G, G^\prime$ and $f$ are fixed as in the statement of Proposition \ref{class::prop:construct_isomorphism_phi_from_f}. Further, $x \in U \cup V$ is a fixed vertex in $G$ and $I(x) = \{f(e) \mid x \in e\}$ is the set from Proposition \ref{class::prop:construct_isomorphism_phi_from_f}. Let us start with a technical lemma.

\begin{lemma} 
    \label{class::lemma:properties_of_I(x)}
    Let $v_1^\prime, v_2^\prime \in V^\prime$ be two distinct vertices of $G^\prime$ with common neighbors $u_1^\prime, \dots, u_n^\prime \in U^\prime$. Assume
    \begin{align*}
        \{u_1^\prime, v_1^\prime\} \in I(x),
        \quad \text{and} \quad 
        \{u_1^\prime, v_2^\prime\} \not \in I(x).
    \end{align*}
    Then for all $i \leq n$ we have
    \begin{align*}
        \{u_i^\prime, v_1^\prime\} \in I(x),
        \quad \text{and} \quad 
        \{u_i^\prime, v_2^\prime\} \not \in I(x).
    \end{align*}
\end{lemma}

\begin{proof}
    If $n \leq 1$ there is nothing to show. So assume $n > 1$ and let $i > 1$. The following is a subgraph of $G^\prime$
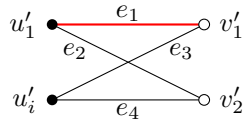
    \begin{center}
        \begin{tikzpicture}
            \draw[color=red, thick] (-1,2) -- (1, 2);
            \draw (-1,1) -- (1, 1);
            \draw (-1,2) -- (1, 1);
            \draw (-1,1) -- (1, 2);
            \node[label=90:$e_1$] at (0, 1.8) {};
            \node[label=-90:$e_4$] at (0, 1.2) {};
            \node[label=90:$e_2$] at (-.7, 1.3) {};
            \node[label=90:$e_3$] at (.7, 1.3) {};
            \node[label=180:$u_1^\prime$, fill=black, circle, inner sep=.05cm] at (-1, 2) {};
            \node[label=180:$u_i^{\prime}$, fill=black, circle, inner sep=.05cm] at (-1, 1) {};
            \node[label=0:$v^\prime_1$, fill=white, draw=black, circle, inner sep=.05cm] at (1, 2) {};
            \node[label=0:$v^\prime_2$, fill=white, draw=black, circle, inner sep=.05cm] at (1, 1) {};
        \end{tikzpicture}
    \end{center}
    where the edge $e_1$ (highlighted in red) is in $I(x)$. This means that $x \in f^{-1}(e_1)$. Since $f$ has property \eqref{reps::eq:f_preserves_K22_subgraphs} we know
    \begin{align*}
        G\left(f^{-1}(e_1), f^{-1}(e_2), f^{-1}(e_3), f^{-1}(e_4)\right) \cong K_{2,2}.
    \end{align*}
    Hence, either $e_2$ or $e_3$ is in $I(x)$ while $e_4 = \{u_i^\prime, v_2^\prime\} \not \in I(x)$. By assumption, however, we have $e_3 = \{u_1^\prime, v_2^\prime\} \not \in I(x)$. Therefore, it follows immediately $e_2 = \{u_i^\prime, v_1^\prime\} \in I(x)$ and this concludes the proof.
\end{proof}

\begin{lemma} 
    \label{class::lemma:Px_is_a_projections}
    The element $P_x = \sum_{\{u^\prime, v^\prime\} \in I(x)} \hat{p}_{u^\prime} \hat{p}_{v^\prime} \in C^\ast(G^\prime)$ from Proposition \ref{class::prop:construct_isomorphism_phi_from_f} is a projection. Further, the sets $\{P_u \mid u \in U\}$ and $\{P_v \mid v \in V\}$ form a partition of unity in $C^\ast(G^\prime)$, respectively, i.e. they satisfy relation \eqref{bga::eq:partition_relation}.
\end{lemma}

\begin{proof}
    Observe for any $x \in U \cup V$
    \begin{align*}
        P_x P_x^\ast 
            &= \left( \sum_{\{u^\prime, v^\prime\} \in I(x)} \hat{p}_{u^\prime} \hat{p}_{v^\prime} \right)
            \left( \sum_{\{u^{\prime\prime}, v^{\prime\prime}\} \in I(x)} \hat{p}_{u^{\prime\prime}} \hat{p}_{v^{\prime\prime}} \right)^\ast \\
            &= \left( \sum_{\{u^\prime, v^\prime\} \in I(x)} \hat{p}_{u^\prime} \hat{p}_{v^\prime} \right)
            \left( \sum_{\{u^{\prime\prime}, v^{\prime\prime}\} \in I(x)} \hat{p}_{v^{\prime\prime}} \hat{p}_{u^{\prime\prime}} \right) \\
            &=  \sum_{\{u^\prime, v^\prime\}, \{u^{\prime\prime}, v^{\prime\prime}\} \in I(x)} \hat{p}_{u^\prime} \hat{p}_{v^\prime} \hat{p}_{v^{\prime\prime}} \hat{p}_{u^{\prime\prime}} \\
            &=  \sum_{\{u^\prime, v^\prime\}, \{u^{\prime\prime}, v^{\prime}\} \in I(x)} \hat{p}_{u^\prime} \hat{p}_{v^\prime} \hat{p}_{u^{\prime\prime}},
    \end{align*}
    where we use \eqref{bga::eq:partition_relation} for the generators of $C^\ast(G^\prime)$ in the last step.
    Let $u^\prime$ and $u^{\prime\prime} \in U^\prime$ be fixed. We claim
    \begin{align*}
        \sum_{v^\prime: \{u^\prime, v^\prime\}, \{u^{\prime\prime}, v^{\prime}\} \in I(x)} \hat{p}_{u^\prime} \hat{p}_{v^\prime} \hat{p}_{u^{\prime\prime}}
        &= \left( \sum_{v^\prime: \{u^\prime, v^\prime\} \in I(x)} \hat{p}_{u^\prime} \hat{p}_{v^\prime} \right) \hat{p}_{u^{\prime\prime}}.
    \end{align*}
    To prove the claim let us distinguish the following cases.

    Case 1. Assume $u^\prime = u^{\prime\prime}$. Then the claim is immediate, for $\{u^\prime, v^\prime\} \in I(x)$ and $\{u^{\prime\prime}, v^{\prime}\} \in I(x)$ are equivalent.

    Case 2. Assume $u^\prime \neq u^{\prime\prime}$. Further, suppose that there is a common neighbor $v_0^\prime$ of $u^\prime$ and $u^{\prime\prime}$ such that 
    \begin{align*}
        \{u^\prime, v_0^\prime\} \in I(x),
        \quad \text{and} \quad 
        \{u^{\prime\prime}, v_0^\prime\} \not \in I(x).
    \end{align*}
    Then, by Lemma \ref{class::lemma:properties_of_I(x)} the same holds for any common neighbor of $u^\prime$ and $u^{\prime\prime}$ instead of $v_0^\prime$. Hence, we have on the one hand
    \begin{align*}
        \sum_{v^\prime: \{u^\prime, v^\prime\}, \{u^{\prime\prime}, v^{\prime}\} \in I(x)} \hat{p}_{u^\prime} \hat{p}_{v^\prime} \hat{p}_{u^{\prime\prime}}
        = \sum_{\emptyset} &= 0,
    \end{align*}
    and on the other hand
    \begin{align*}
         \left( \sum_{v^\prime: \{u^\prime, v^\prime\} \in I(x)} \hat{p}_{u^\prime} \hat{p}_{v^\prime} \right) \hat{p}_{u^{\prime\prime}}
         &= \hat{p}_{u^\prime} \left( \sum_{v^\prime \in \mathcal{N}(u^\prime) \cap \mathcal{N}(u^{\prime\prime})} \hat{p}_{v^\prime} \right) \hat{p}_{u^{\prime\prime}} \\
         &= \hat{p}_{u^\prime} \hat{p}_{u^{\prime\prime}} \\
         &= 0,
    \end{align*}
    where we use in the first step that all common neighbors of $u^\prime$ and $u^{\prime\prime}$ are in $I(x)$ and products over non-neighbors vanish by \eqref{bga::eq:orthogonality_relation}, and in the last step relation \eqref{bga::eq:partition_relation} for the generators of $C^\ast(G^\prime)$.

    Case 3. Finally, assume that neither of the previous cases applies. 
    Then $u^\prime \neq u^{\prime\prime}$ and for any common neighbor $v^\prime$ of $u^\prime$ and $u^{\prime\prime}$ we have 
    $$
        \{u^\prime, v^\prime\} \in I(x) \implies \{u^{\prime\prime}, v^\prime\} \in I(x).
    $$
    Using the relations \eqref{bga::eq:orthogonality_relation} for the generators of $C^\ast(G^\prime)$ one obtains 
    \begin{align*}
        \sum_{v^\prime: \{u^\prime, v^\prime\}, \{u^{\prime\prime}, v^{\prime}\} \in I(x)} \hat{p}_{u^\prime} \hat{p}_{v^\prime} \hat{p}_{u^{\prime\prime}}
        &= \sum_{v^\prime: \{u^\prime, v^\prime\} \in I(x), \{u^{\prime\prime}, v^{\prime}\} \in E^\prime} \hat{p}_{u^\prime} \hat{p}_{v^\prime} \hat{p}_{u^{\prime\prime}} \\
        &= \left( \sum_{v^\prime: \{u^\prime, v^\prime\} \in I(x)} \hat{p}_{u^\prime} \hat{p}_{v^\prime} \right) \hat{p}_{u^{\prime\prime}}.
    \end{align*}
    
    Altogether we conclude
    \begin{align*}
        P_x P_x^\ast
        &=  \sum_{\{u^\prime, v^\prime\}, \{u^{\prime\prime}, v^{\prime}\} \in I(x)} \hat{p}_{u^\prime} \hat{p}_{v^\prime} \hat{p}_{u^{\prime\prime}} \\
        &= \sum_{u^\prime, u^{\prime\prime} \in U^\prime} \left( \sum_{v': \{u^\prime, v^\prime\}, \{u^{\prime\prime}, v^{\prime}\} \in I(x)} \hat{p}_{u^\prime} \hat{p}_{v^\prime} \hat{p}_{u^{\prime\prime}} \right) \\
        &= \sum_{u^\prime, u^{\prime\prime} \in U^\prime} \left( \left( \sum_{v^\prime: \{u^\prime, v^\prime\} \in I(x)} \hat{p}_{u^\prime} \hat{p}_{v^\prime} \right) \hat{p}_{u^{\prime\prime}} \right) \\
        &= \left( \sum_{\{u^\prime, v^\prime\} \in I(x)} \hat{p}_{u^\prime} \hat{p}_{v^\prime} \right) \left( \sum_{u^{\prime\prime} \in U^\prime} \hat{p}_{u^{\prime\prime}} \right) \\
        &= P_x,
    \end{align*}
    where we use \eqref{bga::eq:partition_relation} in the last step. This shows that $P_x$ is a projection in $C^\ast(G^\prime)$ since one obtains directly $P_x^\ast = \left( P_x P_x^\ast \right)^\ast = P_x P_x^\ast = P_x$.
    Finally, observe
    \begin{align*}
        \sum_{u \in U} P_u 
            &= \sum_{u \in U} \left( \sum_{\{u^\prime, v^\prime\} \in I(u)} \hat{p}_{u^\prime} \hat{p}_{v^\prime} \right) \\
            &= \sum_{\{u^\prime, v^\prime\} \in E^\prime} \hat{p}_{u^\prime} \hat{p}_{v^\prime} \\
            &= \left( \sum_{u^\prime \in U^\prime} \hat{p}_{u^\prime} \right) \left( \sum_{v^\prime \in V^\prime} \hat{p}_{v^\prime} \right) \\
            &= 1 \cdot 1 = 1,
    \end{align*}
    where we use \eqref{bga::eq:orthogonality_relation} in the second-last step and \eqref{bga::eq:partition_relation} in the last step.
    This shows that the set $\{P_u \mid u \in U\}$ is a partition of unity in $C^\ast(G^\prime)$. The argument works analogously for $\{P_v \mid v \in V\}$.
\end{proof}

\begin{lemma} 
    \label{class::lemma:Px_are_orthogonal}
    Let $P_x = \sum_{\{u^\prime, v^\prime\} \in I(x)} \hat{p}_{u^\prime} \hat{p}_{v^\prime} \in C^\ast(G^\prime)$ be the element from Proposition \ref{class::prop:construct_isomorphism_phi_from_f}. For any $u \in U$ and $v \in V$ we have $P_u P_v = 0$ unless $u \sim v$, i.e. the relations \eqref{bga::eq:orthogonality_relation} are satisfied.
\end{lemma}

\begin{proof}
    By the previous Lemma \ref{class::lemma:Px_is_a_projections} the $P_x$ are projections. Observe 
    \begin{align*}
        P_u P_v 
            &= P_u P_v^\ast \\
            &= \left( \sum_{\{u^\prime, v^\prime\} \in I(u)} \hat{p}_{u^\prime} \hat{p}_{v^\prime} \right)
                \left( \sum_{\{u^{\prime\prime}, v^{\prime\prime}\} \in I(v)} \hat{p}_{v^{\prime\prime}} \hat{p}_{u^{\prime\prime}} \right) \\
            &= \sum_{\{u^\prime, v^\prime\} \in I(u), \{u^{\prime\prime}, v^{\prime\prime}\} \in I(v)} \hat{p}_{u^\prime} \hat{p}_{v^\prime} \hat{p}_{v^{\prime\prime}} \hat{p}_{u^{\prime\prime}} \\
            &= \sum_{\{u^\prime, v^\prime\} \in I(u), \{u^{\prime\prime}, v^{\prime}\} \in I(v)} \hat{p}_{u^\prime} \hat{p}_{v^\prime} \hat{p}_{u^{\prime\prime}},
    \end{align*}
    where we use \eqref{bga::eq:orthogonality_relation} in the last step.
    Assume that there are some $\{u^\prime, v^\prime\} \in I(u)$ and $\{u^{\prime\prime}, v^\prime\} \in I(v)$ such that 
    \begin{align*}
        \hat{p}_{u^\prime} \hat{p}_{v^\prime} \hat{p}_{u^{\prime\prime}} \neq 0.
    \end{align*}
    We distinguish two cases. First, if $u^\prime = u^{\prime\prime}$, then we have
    \begin{align*}
        \{u^\prime, v^\prime\} \in I(u) \cap I(v) = \{f(e) \mid e = \{u,v\} \in E\},
    \end{align*}
    and therefore $u \sim v$.

    Second, let us assume $u^\prime \neq u^{\prime\prime}$. In view of Lemma \ref{bga::lemma:nontrivial_p_x2 p_y p_x2_and_K22_subgraph} there must be some $v^{\prime\prime} \neq v^\prime$ such that the following is a subgraph of $G^\prime$.
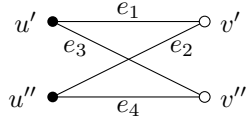
    \begin{center}
        \begin{tikzpicture}
            \draw (-1,2) -- (1, 2);
            \draw (-1,1) -- (1, 1);
            \draw (-1,2) -- (1, 1);
            \draw (-1,1) -- (1, 2);
            \node[label=90:$e_1$] at (0, 1.8) {};
            \node[label=-90:$e_4$] at (0, 1.2) {};
            \node[label=90:$e_3$] at (-.7, 1.3) {};
            \node[label=90:$e_2$] at (.7, 1.3) {};
            \node[label=180:$u^\prime$, fill=black, circle, inner sep=.05cm] at (-1, 2) {};
            \node[label=180:$u^{\prime\prime}$, fill=black, circle, inner sep=.05cm] at (-1, 1) {};
            \node[label=0:$v^\prime$, fill=white, draw=black, circle, inner sep=.05cm] at (1, 2) {};
            \node[label=0:$v^{\prime\prime}$, fill=white, draw=black, circle, inner sep=.05cm] at (1, 1) {};
        \end{tikzpicture}
    \end{center}
    Using that $f$ satisfies \eqref{reps::eq:f_preserves_K22_subgraphs} it follows that the subgraph 
    $$ G(f^{-1}(e_1), f^{-1}(e_2), f^{-1}(e_3), f^{-1}(e_4)) \subset G $$ 
    is isomorphic to $K_{2,2}$. Because of \eqref{reps::eq:f_preserves_adjacency} the edges $f^{-1}(e_1)$ and $f^{-1}(e_2)$ are adjacent. Using $u \in f^{-1}(e_1) \cap U$ and $v \in f^{-1}(e_2) \cap V$ one easily checks that the subgraph $G(f^{-1}(e_1), f^{-1}(e_2))$ can only take two different forms which are sketched in Figure \ref{class::fig:G(f^-1(e_1), f^-1(e_2))}. In both cases, one sees $u \sim v$.
    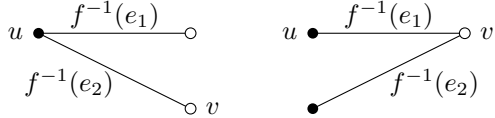
\begin{figure}[htb]
        \begin{tikzpicture}
            \draw (-1,2) -- (1, 2);
            \draw (-1,2) -- (1, 1);
            \node[label=90:$f^{-1}(e_1)$] at (0, 1.8) {};
            \node[label=90:$f^{-1}(e_2)$] at (-.6, .9) {};
            \node[label=180:$u$, fill=black, circle, inner sep=.05cm] at (-1, 2) {};
            \node[fill=white, draw=black, circle, inner sep=.05cm] at (1, 2) {};
            \node[label=0:$v$, fill=white, draw=black, circle, inner sep=.05cm] at (1, 1) {};
        \end{tikzpicture}
        $\quad$
        \begin{tikzpicture}
            \draw (-1,2) -- (1, 2);
            \draw (-1,1) -- (1, 2);
            \node[label=90:$f^{-1}(e_1)$] at (0, 1.8) {};
            \node[label=90:$f^{-1}(e_2)$] at (.6, .9) {};
            \node[label=180:$u$, fill=black, circle, inner sep=.05cm] at (-1, 2) {};
            \node[ fill=black, circle, inner sep=.05cm] at (-1, 1) {};
            \node[label=0:$v$, fill=white, draw=black, circle, inner sep=.05cm] at (1, 2) {};
        \end{tikzpicture}
        \caption{The two possible forms of the subgraph $G(f^{-1}(e_1), f^{-1}(e_2))$}
        \label{class::fig:G(f^-1(e_1), f^-1(e_2))}
    \end{figure}
\end{proof}

Putting the previous lemmas together we obtain a proof of Proposition \ref{class::prop:construct_isomorphism_phi_from_f}.

\begin{proof}[Proof of Proposition \ref{class::prop:construct_isomorphism_phi_from_f}]
    The preceding lemmas show that the $P_x$ satisfy the relations required for the generators of $C^\ast(G)$. Thus, by the universal property of $C^\ast(G)$ there is a unique $\ast$-homomorphism $\varphi_f: C^\ast(G) \to C^\ast(G^\prime)$ with $\varphi_f(p_x) = P_x$ for all $x \in U \cup V$.     
\end{proof}

Let us summarize what we have seen so far. Given a homeomorphism $\Phi: \mathrm{Spec}_{\leq 2}(C^\ast(G)) \to \mathrm{Spec}_{\leq 2}(C^\ast(G^\prime))$, we obtain from Lemma \ref{reps::lemma:edge_bijection_f_from_homeomorphism_of_spectrum} a bijective map $f: E \to E^\prime$ that satisfies the properties \eqref{reps::eq:f_preserves_K22_subgraphs} and \eqref{reps::eq:f_preserves_adjacency}. Evidently, its inverse map $f^{-1}: E^\prime \to E$ has the same properties. Thus, Proposition \ref{class::prop:construct_isomorphism_phi_from_f} yields two $\ast$-homomorphisms $\varphi_f: C^\ast(G) \to C^\ast(G^\prime)$ and $\varphi_{f^{-1}}: C^\ast(G^\prime) \to C^\ast(G)$ that are defined on the respective generators as follows:
\begin{align}
    \label{class::eq:phi_on_generators}
    \varphi_f: C^\ast(G) \ni p_x &\mapsto P_x = \sum_{\{u^\prime, v^\prime\} \in I(x)} \hat{p}_{u^\prime} \hat{p}_{v^\prime} \in C^\ast(G^\prime), & & x \in U \cup V, \\
    \label{class::eq:psi_on_generators}
    \varphi_{f^{-1}}: C^\ast(G^\prime) \ni \hat{p}_y &\mapsto \hat{P}_y := \sum_{\{u, v\} \in J(y)} p_u p_v \in C^\ast(G), & & y \in U^\prime \cup V^\prime,
\end{align}
where $I(x) = \{f(e) \in E^\prime: x \in e\}$ and $J(x) = \{f^{-1}(e) \in E: y \in e\}$ for $x \in U \cup V$ and $y \in U^\prime \cup V^\prime$, respectively. In the above formulas we implicitly require $u \in U, v \in V$ and $u^\prime \in U^\prime, v^\prime \in V^\prime$.

To prove $C^\ast(G) \cong C^\ast(G^\prime)$ it only remains to show that $\varphi_f$ and $\varphi_{f^{-1}}$ are inverse to each other. In fact, it suffices to show $\varphi_{f^{-1}} \circ \varphi_f = \mathrm{id}_{C^\ast(G)}$ because the other statement $\varphi_f \circ \varphi_{f^{-1}} = \mathrm{id}_{C^\ast(G^\prime)}$ follows by replacing $f$ with $f^{-1}$. We prove this in the following lemma.

\begin{lemma} 
    \label{class::lemma:phi_is_invertible}
    One has $\varphi_{f^{-1}} \circ \varphi_f = \mathrm{id}_{C^\ast(G)}$.
\end{lemma}

\begin{proof}
    It suffices to show 
    \begin{align*}
        \varphi_{f^{-1}}(\varphi_f(p_x)) = p_x,
    \end{align*}
    for all $x \in U \cup V$. Indeed, we only prove this for $x \in U$ because the case $x \in V$ is analogous. Thus, let $u \in U$ be fixed.
    The definitions of $\varphi_{f^{-1}}$ and $\varphi_f$, respectively, yield 
    \begin{align*}
        \varphi_{f^{-1}}(\varphi_f(p_u))
            &= \varphi_{f^{-1}}\left( \sum_{\{u^\prime, v^\prime\} \in I(u)} \hat{p}_{u^\prime} \hat{p}_{v^\prime} \right) \\
            &= \sum_{\{u^\prime, v^\prime\} \in I(u)} \left( \sum_{\{u_1, v_1\} \in J(u^\prime)} p_{u_1} p_{v_1} \right) \left( \sum_{\{u_2, v_2\} \in J(v^\prime)} p_{u_2} p_{v_2} \right).
    \end{align*}
    Using $\varphi_{f^{-1}}(\hat{p}_{u^\prime}) = \left(\varphi_{f^{-1}}(\hat{p}_{u^\prime})\right)^\ast = \sum_{\{u_1, v_1\} \in J(u^\prime)} p_{v_1} p_{u_1}$ together with \eqref{bga::eq:orthogonality_relation} we can continue the computation as follows:
    \begin{align*}
        \varphi_{f^{-1}}(\varphi_f(p_u))
            &= \sum_{\{u^\prime, v^\prime\} \in I(u)} \left( \sum_{\{u_1, v_1\} \in J(u^\prime)}  p_{v_1} p_{u_1}\right) \left( \sum_{\{u_2, v_2\} \in J(v^\prime)} p_{u_2} p_{v_2} \right) \\
            &= \sum_{\{u^\prime, v^\prime\} \in I(u)} \left( \sum_{\{u_1, v_1\} \in J(u^\prime), \{u_1, v_2\} \in J(v^\prime)} p_{v_1} p_{u_1} p_{v_2} \right).
    \end{align*}
    We claim
    \begin{align*} 
        &\left\{(v_1, u_1, v_2) \left| \;
            \begin{aligned}
                &\{u_1, v_1\} \in J(u^\prime) \\
                &\{u_1, v_2\} \in J(v^\prime)
            \end{aligned} \;
                \text{ for some } \{u^\prime, v^\prime\} \in I(u)
                \text{ with } p_{v_1} p_{u_1} p_{v_2} \neq 0 
            \right. \right\} \\
        &=\{(v_1, u, v_2): V \ni v_1 \sim u \sim v_2 \in V \text{ with } p_{v_1} p_u p_{v_2} \neq 0 \},
    \end{align*}
    where $u$ remains the fixed vertex from above.
    First, assume that $(v_1, u, v_2)$ is in the set on the right-hand side, i.e. 
    $$ V \ni v_1 \sim u \sim v_2 \in V \quad \text{with} \quad p_{v_1} p_u p_{v_2} \neq 0. $$ 
    There are two possibilities. If $v_1 = v_2$, then setting $\{u^\prime, v^\prime\} := f(\{u, v_1\})$ one easily verifies $\{u, v_1\} \in J(u^\prime) \cap J(v^\prime)$ while $\{u^\prime, v^\prime\} \in I(u)$. Thus, the triple $(v_1, u, v_2)$ is in the set on the left-hand side.
    
    Next, assume $v_1 \neq v_2$. As $p_{v_1} p_u p_{v_2}$ is non-zero, Lemma \ref{bga::lemma:nontrivial_p_x2 p_y p_x2_and_K22_subgraph} yields some vertex $u_2 \in U$ such that the graph from Figure \ref{class::fig:G(u, v_1, v_2, u_2)} is a subgraph of $G$.
    \begin{figure}[htb]
        \begin{tikzpicture}
            \draw[thick, color=red] (-1,2) -- (1, 2);
            \draw (-1,1) -- (1, 1);
            \draw[thick, color=red] (-1,2) -- (1, 1);
            \draw (-1,1) -- (1, 2);
            \node[label=90:$e_1$] at (0, 1.8) {};
            \node[label=-90:$e_4$] at (0, 1.2) {};
            \node[label=90:$e_2$] at (-.7, 1.3) {};
            \node[label=90:$e_3$] at (.7, 1.3) {};
            \node[label=180:$u$, fill=black, circle, inner sep=.05cm] at (-1, 2) {};
            \node[label=180:$u_2$, fill=black, circle, inner sep=.05cm] at (-1, 1) {};
            \node[label=0:$v_1$, fill=white, draw=black, circle, inner sep=.05cm] at (1, 2) {};
            \node[label=0:$v_2$, fill=white, draw=black, circle, inner sep=.05cm] at (1, 1) {};
        \end{tikzpicture}
        \caption{The subgraph $G(u, v_1, v_2, u_2)$}
        \label{class::fig:G(u, v_1, v_2, u_2)}
    \end{figure}
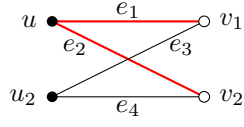
    As $e_1$ and $e_2$ are adjacent in $G(e_1, e_2, e_3, e_4)$, by Property \eqref{reps::eq:f_preserves_adjacency} the same holds in $G^\prime(f(e_1), f(e_2), f(e_3), f(e_4))$. Thus, the subgraph $G^\prime(f(e_1), f(e_2))$ must take one of the two forms shown in Figure \ref{class::fig:G(f(e_1), f(e_2))}.
    \begin{figure}[htb]
        \begin{tikzpicture}
            \draw (-1,2) -- (1, 2);
            \draw (-1,2) -- (1, 1);
            \node[label=90:$f(e_1)$] at (0, 1.8) {};
            \node[label=90:$f(e_2)$] at (-.6, .9) {};
            \node[label=180:$u^\prime$, fill=black, circle, inner sep=.05cm] at (-1, 2) {};
            \node[fill=white, draw=black, circle, inner sep=.05cm] at (1, 2) {};
            \node[label=0:$v^\prime$, fill=white, draw=black, circle, inner sep=.05cm] at (1, 1) {};
        \end{tikzpicture}
        $\quad$
        \begin{tikzpicture}
            \draw (-1,2) -- (1, 2);
            \draw (-1,1) -- (1, 2);
            \node[label=90:$f(e_1)$] at (0, 1.8) {};
            \node[label=90:$f(e_2)$] at (.6, .9) {};
            \node[label=180:$u^\prime$, fill=black, circle, inner sep=.05cm] at (-1, 2) {};
            \node[ fill=black, circle, inner sep=.05cm] at (-1, 1) {};
            \node[label=0:$v^\prime$, fill=white, draw=black, circle, inner sep=.05cm] at (1, 2) {};
        \end{tikzpicture}
        \caption{The two possible forms of the subgraph $G^\prime(f(e_1), f(e_2))$}
        \label{class::fig:G(f(e_1), f(e_2))}
    \end{figure} 
    Choosing the vertices $u^\prime \in f(e_1)$ and $v^\prime \in f(e_2)$ as in the figure one easily checks $\{u^\prime, v^\prime\} \in I(u)$ as well as $e_1 = \{u, v_1\} \in J(u^\prime)$ and $e_2 = \{u, v_2\} \in J(v^\prime)$. Thus, the triple $(v_1, u, v_2)$ is in the set on the left-hand side, and this concludes the proof of the inclusion $(\mathrm{lhs}) \supseteq (\mathrm{rhs})$. 

    For the other inclusion, assume that $(v_1, u_1, v_2)$ is in the set on the left-hand side, i.e.
    \begin{align*}
            \begin{aligned}
                &\{u_1, v_1\} \in J(u^\prime) \\
                &\{u_1, v_2\} \in J(v^\prime)
            \end{aligned} \;
                \text{ for some } \{u^\prime, v^\prime\} \in I(u)
                \text{ with } p_{v_1} p_{u_1} p_{v_2} \neq 0.
    \end{align*}
    It suffices to show $u_1 = u$ since $p_{v_1} p_{u_1} p_{v_2} \neq 0$ implies $v_1 \sim u_1 \sim v_2$ (using \eqref{bga::eq:orthogonality_relation}) and thus the triple $(v_1, u_1, v_2)$ is in the set on the right-hand side.

    If $v_1 = v_2$, then we obtain immediately $f(\{u_1, v_1\}) = \{u^\prime, v^\prime\} \in I(u)$ and this yields $u^\prime = u$. 

    Next, assume $v_1 \neq v_2$. In view of Lemma \ref{bga::lemma:nontrivial_p_x2 p_y p_x2_and_K22_subgraph} there exists some $u_2 \in U$ such that the graph from Figure \ref{class::fig:G(u_1, v_1, v_2, u_2)} is a subgraph of $G$. 
    \begin{figure}[htb]
        \begin{tikzpicture}
            \draw[thick, color=red] (-1,2) -- (1, 2);
            \draw (-1,1) -- (1, 1);
            \draw[thick, color=red] (-1,2) -- (1, 1);
            \draw (-1,1) -- (1, 2);
            \node[label=90:$e_1$] at (0, 1.8) {};
            \node[label=-90:$e_4$] at (0, 1.2) {};
            \node[label=90:$e_2$] at (-.7, 1.3) {};
            \node[label=90:$e_3$] at (.7, 1.3) {};
            \node[label=180:$u_1$, fill=black, circle, inner sep=.05cm] at (-1, 2) {};
            \node[label=180:$u_2$, fill=black, circle, inner sep=.05cm] at (-1, 1) {};
            \node[label=0:$v_1$, fill=white, draw=black, circle, inner sep=.05cm] at (1, 2) {};
            \node[label=0:$v_2$, fill=white, draw=black, circle, inner sep=.05cm] at (1, 1) {};
        \end{tikzpicture}
        \caption{The graph $G(u_1, v_1, v_2, u_2)$}
        \label{class::fig:G(u_1, v_1, v_2, u_2)}
    \end{figure}
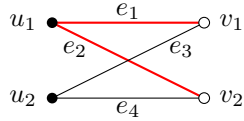
    where $e_1 \in J(u^\prime)\; ( \Leftrightarrow u^\prime \in f(e_1))$, and $e_2 \in J(v^\prime)\; ( \Leftrightarrow v^\prime \in f(e_2))$.
    Again by Property \eqref{reps::eq:f_preserves_adjacency} the subgraph $G^\prime(f(e_1), f(e_2))$ must take one of the two forms shown in Figure \ref{class::fig:G(f(e_1), f(e_2))}. The vertices $u^\prime$ and $v^\prime$ must be as depicted since $e_1 \in J(u^\prime)$ and $e_2 \in J(v^\prime)$. Now, one sees that either $f(e_1) = \{u^\prime, v^\prime\}$ or $f(e_2) = \{u^\prime, v^\prime\}$. Then the assumption $\{u^\prime, v^\prime\} \in I(u)$ implies $f(e_1) \in I(u)$ or $f(e_2) \in I(u)$. In any event, it follows that $u_1 = u$.

    Putting everything together we obtain
    \begin{align*}
        \varphi_{f^{-1}}(\varphi_f(p_u))
            &= \sum_{\{u^\prime, v^\prime\} \in I(u)} \left( \sum_{\{u_1, v_1\} \in J(u^\prime), \{u_1, v_2\} \in J(v^\prime)} p_{v_1} p_{u_1} p_{v_2} \right) \\
            &= \sum_{V \ni v_1 \sim u \sim v_2 \in V} p_{v_1} p_u p_{v_2} \\
            &= \sum_{v_1, v_2 \in V} p_{v_1} p_u p_{v_2} \\
            &= \left( \sum_{v_1 \in V} p_{v_1} \right) p_u \left( \sum_{v_2 \in V} p_{v_2} \right) \\
            &= p_u,
    \end{align*}
    where we use the claim for the step from the first to the second line. The last two equalities follows from \eqref{bga::eq:orthogonality_relation} and \eqref{bga::eq:partition_relation}, respectively.
\end{proof}

Finally, we can state the main result of this section. This theorem was mentioned in the introduction as Theorem \ref{intro::thm:classification_of_bipartite_graph_algebras}.

\begin{thm}
    \label{class::thm:classification_of_bipartite_graph_algebras}
    We have 
    \begin{align*}
        C^\ast(G) \cong C^\ast(G^\prime) 
            \quad \Leftrightarrow \quad 
        \mathrm{Spec}_{\leq 2}(C^\ast(G)) \cong \mathrm{Spec}_{\leq 2}(C^\ast(G^\prime)).
    \end{align*}
\end{thm}

\begin{proof}
    The implication from left to right is clear. So let $\Phi: \mathrm{Spec}_{\leq 2}(C^\ast(G)) \to \mathrm{Spec}_{\leq 2}(C^\ast(G^\prime))$ be a homeomorphism. One obtains an induced bijective map $f: E \to E^\prime$ which satisfies \eqref{reps::eq:f_preserves_K22_subgraphs} and \eqref{reps::eq:f_preserves_adjacency} as discussed in Lemma \ref{reps::lemma:edge_bijection_f_from_homeomorphism_of_spectrum}. The same properties hold for the inverse map $f^{-1}: E^\prime \to E$. By Proposition \ref{class::prop:construct_isomorphism_phi_from_f} we obtain the $\ast$-homomorphisms $\varphi_f$ and $\varphi_{f^{-1}}$ described in \eqref{class::eq:phi_on_generators} and \eqref{class::eq:psi_on_generators}. By Lemma \ref{class::lemma:phi_is_invertible}, $\varphi_{f^{-1}}$ is a left-inverse of $\varphi$, and by replacing $f$ with $f^{-1}$ one gets that $\varphi_{f}$ is a left-inverse of $\varphi_{f^{-1}}$. Consequently, $\varphi_f$ is a $\ast$-isomorphism.
\end{proof}

\printbibliography

\end{document}